\documentclass{article}

\usepackage{amsmath,amssymb,amsthm,graphics,tikz, stmaryrd}
\usepackage{comment}
\usepackage{lineno}

\title{Hierarchical formula classes with respect to semi-classical prenex normalization}
\author{Makoto Fujiwara\footnote{Email: makotofujiwara@rs.tus.ac.jp}
\footnote{Department of Applied Mathematics, Faculty of Science Division I, Tokyo University of Science, 1-3 Kagurazaka, Shinjuku-ku, Tokyo 162-8601, Japan.}
and Taishi Kurahashi\footnote{Email: kurahashi@people.kobe-u.ac.jp}
\footnote{Graduate School of System Informatics,
Kobe University,
1-1 Rokkodai, Nada, Kobe 657-8501, Japan.}}
\date{}

\theoremstyle{plain}
\newtheorem{thm}{Theorem}[section]
\newtheorem*{thm*}{Theorem}
\newtheorem{lem}[thm]{Lemma}
\newtheorem{prop}[thm]{Proposition}
\newtheorem{cor}[thm]{Corollary}

\newtheorem*{notation}{Notation}

\theoremstyle{definition}
\newtheorem{defn}[thm]{Definition}

\newtheorem{remark}[thm]{Remark}

\newcommand{\FV}[1]{\mathrm{ FV} \left({#1}\right)}
\newcommand{\DT}{\rightsquigarrow}
\newcommand{\DTA}{{\rightsquigarrow^*}}

\newcommand{\qz}{{\,\, \rightsquigarrow^{*}_{0}\,\,}}

\newcommand{\QF}[1]{{#1}_{\mathrm{qf}}}

\newcommand{\PA}{\mathsf{PA}}
\newcommand{\HA}{\mathsf{HA}}

\newcommand{\DNE}[1]{{#1}\text{-}\mathrm{DNE}}
\newcommand{\LEM}[1]{{#1}\text{-}\mathrm{LEM}}

\newcommand{\F}{\mathrm{F}}
\newcommand{\PF}{\mathrm{P}}
\newcommand{\U}{\mathrm{U}}
\newcommand{\E}{\mathrm{E}}
\newcommand{\C}{\mathrm{C}}
\newcommand{\R}{\mathcal{U}}
\newcommand{\J}{\mathcal{E}}

\newcommand{\D}{\mathcal{D}}

\newcommand{\N}{\mathbb{N}}
\newcommand{\vp}{\varphi}

\newcommand{\qn}{{\, \rightsquigarrow^{*}_{n}\,}}
\newcommand{\ol}[1]{\overline{#1}}
\newcommand{\br}[1]{\llbracket{#1}\rrbracket}
\newcommand{\PNFT}[2]{\mathrm{PNFT}_{T} \left( {#1}, {#2} \right) }

\begin{document}

\maketitle

\begin{abstract}
In \cite{FK24}, the authors formalized the standard transformation procedure for prenex normalization of first-order formulas and showed that the classes $\E_k$ and $\U_k$ introduced in Akama et al. \cite{ABHK04} are exactly the classes induced by $\Sigma_k$ and $\Pi_k$ respectively via the transformation procedure.
In that sense, the classes $\E_k$ and $\U_k$ correspond to  $\Sigma_k$ and $\Pi_k$ based on classical logic respectively.
On the other hand, some transformations of the prenex normalization are not possible in constructive theories.
In this paper, we introduce new classes $\J_k^n$ and $\R_k^n$ of first-order formulas with two parameters $k$ and $n$, and show that they are exactly the classes induced by $\Sigma_k$ and $\Pi_k$ respectively according to the $n$-th level semi-classical prenex normalization, which is obtained by the prenex normalization in \cite{FK24} with some restriction to the introduced classes of degree $n$.
In particular, the latter corresponds to possible transformations in intuitionistic arithmetic augmented with the law-of-excluded-middle schema restricted to formulas of $\Sigma_n$-form.
In fact, if $ n\geq k$, our classes  $\J_k^n$ and $\R_k^n$ are identical with the cumulative variants $\E^+_k$ and $\U^+_k$ of $\E_k$ and $\U_k$ respectively.
In this sense, our classes are refinements of $\E^+_k$ and $\U^+_k$ with respect to the prenex normalization from the semi-classical perspective.
\end{abstract}



\section{Introduction}
The prenex normal form theorem states that for any first-order theory based on classical logic, every formula is equivalent (over the theory in question) to some formula in prenex normal form (cf.~\cite[pp.~160--161]{End01}).
On the other hand, it does not hold for constructive theories in general.
Therefore the arithmetical hierarchical classes $\Sigma_k$ and $\Pi_k$, which are based on prenex formulas, do not make sense for constructive theories.
Based on this fact, several kinds of hierarchical classes corresponding to $\Sigma_k$ and $\Pi_k$ from the constructive viewpoint have been introduced and studied from different perspectives respectively (cf. \cite{Mints68, Leiv81, Burr00, Fle10, SUZ17, BNI19}).

In \cite{FK24}, the authors formalized the standard transformation procedure for prenex normalization of first-order formulas and called it {\it prenex normalization}  (see Table \ref{table: PN}).
The prenex normalization is a reduction procedure without any reference to the notion of derivability, and it is based on the standard proof of prenex normal form theorem for classical theories (see \cite[Section 1]{FK24}).
Then they showed that the classes $\E_k$ and $\U_k$ introduced in Akama et al. \cite{ABHK04} are exactly the classes induced by $\Sigma_k$ and $\Pi_k$ respectively via the prenex normalization (see \cite[Theorem 16]{FK24}).
Therefore, the classes $\E_k$ and $\U_k$ correspond to $\Sigma_k$ and $\Pi_k$ based on classical logic respectively.
Since some transformations of the prenex normalization are not possible in constructive theories, however, the classes do not make sense from the constructive standpoint.
On the other hand, if one restricts the classes of formulas for the prenex normalization, they are possible in some semi-classical theories which are obtained from constructive theories by adding some restricted fragment of classical logic.
For example, the transformation rule $(\to \exists)$ in Table \ref{table: PN} is not derivable in intuitionistic arithmetic $\HA$ in general, but it is derivable for $\delta$ of $\Sigma_n$-form in semi-classical arithmetic $\HA+\LEM{\Sigma_n}$, where $\LEM{\Sigma_n}$ denotes the law-of-excluded-middle schema restricted to formulas of $\Sigma_n$-form.
Based on this sort of idea, in this paper, we introduce new classes $\J_k^n$ and $\R_k^n$ of first-order formulas with two parameters $k$ and $n$, and show that they are exactly the classes induced by (the cumulative variants of) $\Sigma_k$ and $\Pi_k$ respectively according to the $n$-th level semi-classical prenex normalization (see Theorem \ref{thm: characterizations of semi-classical PN}).
The latter is the prenex normalization with some restriction to the introduced classes of degree $n$, and this corresponds to possible transformations in $\HA+\LEM{\Sigma_n}$.
In fact, if $ n\geq k$, the classes $\J_k^n$ and $\R_k^n$ are identical with the cumulative variants $\E^+_k$ and $\U^+_k$ of $\E_k$ and $\U_k$ respectively (see Proposition \ref{prop: k>=n => Ek+=Enk and Uk+=Unk}).
In this sense, our classes $\J_k^n$ and $\R_k^n$ are refinements of the classes $\E^+_k$ and $\U^+_k$ with respect to the prenex normalization from the semi-classical perspective.
They are new kind of classes corresponding to $\Sigma^+_k$ and $\Pi^+_k$ from relativized viewpoints in-between constructive and classical ones with respect to the prenex normalization.

All of our proofs in this paper are purely syntactic, and the proofs contain many case distinctions.

\subsection{Framework}
We work with a standard formulation of first-order theories with all the logical constants $\forall, \exists, \to, \land, \lor$ and $ \perp$ in the language. 
Note that $\neg \varphi$ and $\varphi \leftrightarrow \psi $ are the abbreviations of $(\varphi \to \perp )$ and $(\varphi \to \psi) \land (\psi \to \varphi)$ respectively in our context.
Throughout this paper, let $k$, $n$, $i$ and $m$ denote natural numbers (possibly $0$).

\begin{notation}
For a formula $\vp$, $\FV{\vp}$ denotes the set of all free variables in $\vp$.
Quantifier-free formulas are denoted with subscript ``qf'' as $\QF{\vp}$.
In addition, a non-empty list of variables is denoted with an over-line as $\ol{x}$.
In particular, a list of quantifiers of the same kind is denoted as $\exists \ol{x}$ and  $\forall \ol{x}$ respectively.
For formulas $\vp, \xi$ and $\xi'$, $\vp\br{\xi' / \xi}$ denotes the formula $\vp$ in which ``an occurrence'' of $\xi$ is replaced by $\xi'$ (note that $\vp\br{\xi' / \xi}$ may be different from the substitution of $\xi$ by $\xi'$ in $\vp$).
\end{notation}

The classes $\Sigma_k$ and $\Pi_k$ are defined as follows (cf.~\cite[pp.~142--143]{CK13}):
\begin{itemize}
	\item Let $\Sigma_0 $ and $ \Pi_0$ be the class of all quantifier-free formulas; 
	\item $\Sigma_{k+1} : = \{\exists \ol{x} \, \varphi \mid \varphi \in \Pi_k\}$;
	\item $\Pi_{k+1} : = \{\forall \ol{x} \, \varphi \mid \varphi \in \Sigma_k\}$.
\end{itemize}
Their cumulative variants $\Sigma_k^+$ and $\Pi_k^+$ are defined as follows:
\begin{itemize}
	\item $ \begin{displaystyle}
 \Sigma_{k}^+ : =\Sigma_{k} \cup \bigcup_{i<k} \Sigma_i \cup  \bigcup_{i<k} \Pi_i ;
 \end{displaystyle}$
\item $ \begin{displaystyle}
  \Pi_{k}^+ : =\Pi_{k} \cup \bigcup_{i<k} \Sigma_i \cup  \bigcup_{i<k} \Pi_i
   \end{displaystyle}$.
\end{itemize}
A formula $\varphi$ is in \textbf{prenex normal form} if $\varphi$ is in $\Sigma_k \cup \Pi_k$ for some $k$.
The classes $\E_k$, $\U_k$, $\PF_k$ and $\F_k$ were introduced in \cite{ABHK04}, and the cumulative variants $\E_k^+$, $\U_k^+$, and $\F_k^+$ were introduced and studied in \cite{FK21, FK23, FK24}.
See \cite[Section 2]{FK24} for their precise definitions.

In the context of arithmetic, the hierarchy of logical axioms restricted to the classes $\Sigma_k$ and $\Pi_k$ has been studied extensively (cf. \cite{ABHK04, FK22}).
The logical axioms include the law-of-excluded-middle schema
\[
{\rm LEM}:\, \vp \lor \neg \vp
\]
and the double-negation-elimination schema
\[
{\rm DNE}:\, \neg \neg \vp \to \vp.
\]
For a class $\Gamma$ of formulas, $\LEM{\Gamma}$ and $\DNE{\Gamma}$ denote ${\rm LEM}$ and ${\rm DNE}$ restricted to formulas in $\Gamma$ respectively.

\subsection{A previous work}
In \cite{FK24}, the transformations in Table \ref{table: PN} are called \textit{prenex transformations}.
\begin{table}[h]
\begin{center}
\begin{tabular}{rccc}
$(\exists \to )$: & $\exists x \xi(x) \to \delta$ &$\rightsquigarrow$ & $\forall x(\xi(x) \to \delta)$;\\

$(\forall \to )$: & $\forall x \xi(x) \to \delta$ &$\rightsquigarrow$ & $\exists x(\xi(x) \to \delta)$;\\

$(\to \exists )$: & $\delta \to \exists  x\, \xi(x)$ & $\rightsquigarrow$ & $\exists  x\, (\delta \to \xi(x))$;\\
$(\to \forall )$: & $\delta \to \forall  x\, \xi(x)$ & $\rightsquigarrow$ & $\forall  x\, (\delta \to \xi(x))$;\\
$(\exists \land )$: & $\exists x \, \xi(x) \land \delta$ &$\rightsquigarrow$ & $\exists x\, (\xi(x) \land \delta)$;\\
$(\forall \land  )$: & $\forall x \, \xi(x) \land \delta$ &$\rightsquigarrow$ & $\forall x\, (\xi(x) \land \delta)$;\\
$(\land \exists)$: & $\delta \land \exists x \, \xi(x)$ & $\rightsquigarrow$ & $\exists x\, (\delta \lor \xi(x))$;\\
$(\lor \forall)$: & $\delta \land \forall x \, \xi(x)$ & $\rightsquigarrow$ & $\forall x\, (\delta \land \xi(x))$;\\
$(\exists  \lor)$: & $\exists x \, \xi(x) \lor \delta$ &$\rightsquigarrow$ & $\exists x\, (\xi(x) \lor \delta)$;\\
$(\forall \lor)$: & $\forall x \, \xi(x) \lor \delta$ &$\rightsquigarrow$ & $\forall x\, (\xi(x) \lor \delta)$;\\
$(\lor \exists)$: & $\delta \lor \exists x \, \xi(x)$ & $\rightsquigarrow$ & $\exists\, (\delta \lor \xi(x))$;\\
$(\lor \forall)$: & $\delta \lor \forall x \, \xi(x)$ & $\rightsquigarrow$ & $\forall x\, (\delta \lor \xi(x))$;\\
$(\exists\text{-var})$:& $ \exists x \xi(x)$ & $\rightsquigarrow$ & $\exists y \xi(y)$;\\
$(\forall \text{-var})$: & $ \forall x \xi(x)$ & $\rightsquigarrow$ & $\forall y \xi(y)$;\\[5pt]
\end{tabular}

\noindent
where $x \notin \FV{\delta}$ and $y$ does not appear in $\xi$.
\end{center}
\caption{Prenex normalization}
\label{table: PN}
\end{table}
Then the relation $\varphi \DTA \psi$ between formulas is defined as $\psi$ is obtained from $\varphi$ by repeating prenex transformations $\DT$ finitely many times to a subformula recursively.
In \cite{FK24}, the authors showed the following:
\begin{itemize}
\item
A formula is in $\E_k^+$ (resp.~$\U_k^+$) if and only if it can be transformed into a formula in $\Sigma_k^+$ (resp.~$\Pi_k^+$) with respect to $\DTA$.
\item
A formula is in $\F_k^+$ if and only if it can be transformed into a formula in $\Sigma_{k+1}^+$ and also into a formula in $\Pi_{k+1}^+$ with respect to $\DTA$.
\item
A formula is in $\E_{k+1}$ (resp.~$\U_{k+1}$) if and only if it can be transformed into a formula in $\Sigma_{k+1}$ (resp.~$\Pi_{k+1}$) but cannot be so for $\Pi_{k+1}^+$ (resp.~$\Sigma_{k+1}^+$) with respect to $\DTA$.
\item
A formula is in $\PF_{k}$ if and only if it can be transformed into a formula in $\Sigma_{k+1}$ and also into  a formula in $\Pi_{k+1}$ but cannot be so for $\Sigma_{k}^+ \cup \Pi_{k}^+$ with respect to $\DTA$.
\end{itemize}
By this observation, the classification of formulas into $\E_k$, $\U_k$ and $\PF_k$ can be visualized as Figure \ref{fig:my_label} (cf. \cite[Section 5]{FK24}).

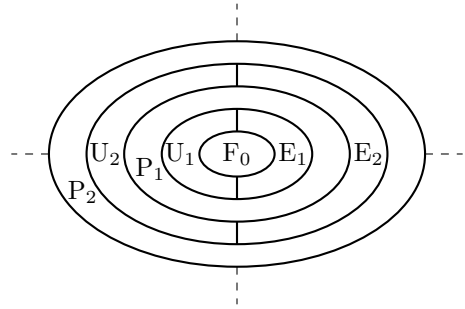
\begin{figure}[ht]
\centering
\begin{tikzpicture}
\node at (0,0) {$\F_0$};
\node at (-0.75,0) {$\U_1$};
\node at (0.75,0) {$\E_1$};
\node at (-1.15, -0.2) {$\PF_1$};
\node at (-1.75,0) {$\U_2$};
\node at (1.75,0) {$\E_2$};
\node at (-2.05, -0.5) {$\PF_2$};

\draw [thick] (0,0)circle [x radius=0.5,y radius=0.3];
\draw [thick] (0,0)circle [x radius=1,y radius=0.6];
\draw [thick] (0,0)circle [x radius=1.5,y radius=0.9];
\draw [thick] (0,0)circle [x radius=2,y radius=1.2];
\draw [thick] (0,0)circle [x radius=2.5,y radius=1.5];

\draw [thick] (0,0.3)--(0,0.6);
\draw [thick] (0,-0.3)--(0,-0.6);
\draw [thick] (0,0.9)--(0,1.2);
\draw [thick] (0,-0.9)--(0,-1.2);
\draw [dashed] (0,1.5)--(0,2);
\draw [dashed] (0,-1.5)--(0,-2);
\draw [dashed] (2.5,0)--(3.0,0);
\draw [dashed] (-2.5,0)--(-3.0,0);
\end{tikzpicture}
\caption{Hierarchical classification of formulas with respect to $\DTA$}
    \label{fig:my_label}
\end{figure}

\section{Semi-classical prenex normalization}
\label{sec: scPN}
Prenex normal form theorem does not hold for intuitionistic theories in general.
In particular, the rules $(\forall \to )$, $(\to \exists )$ and the converses of $(\forall \lor)$ and $(\lor \forall)$ in Table \ref{table: PN} are not derivable in an intuitionistic theory.
In intuitionistic arithmetic $\HA$, however, the rules $(\to \exists )$ and the converses of $(\forall \lor)$ and $(\lor \forall)$  are derivable if formulas $\delta$ are restricted to quantifier-free formulas.
In a general context, they are derivable in an intuitionistic (or semi-classical) theory in which $\delta$ is decidable.

In the following, we construct hierarchical classes $\J^n_k$ and $\R_k^n$ which are intended to be classes corresponding to $\Sigma_k^+$ and $\Pi_k^+$ with assuming the decidability of formulas in $\E_n^+ \cup \U_n^+$.
For example, an intended class $\J^0_{k+1}$ of formulas which are intuitionistically equivalent to some $\Sigma_{k+1}^+$-formulas only with assuming the decidability of quantifier-free formulas is desired to satisfy the following clauses:
\begin{itemize}
\item
Formulas which are already ensured to be equivalent to some formulas in $\Sigma_k^+ \cup \Pi_k^+$ are in $\J^0_{k+1}$.
\item
If $E$ and $E'$ are in $\J^0_{k+1}$, then $\exists x E$ and $E\land E'$ are in $\J^0_{k+1}$;
\item 
If $E$ is in $\J^0_{k+1}$ and $E_1$ is in $\J^0_1$, then $\E \lor \E_1$ and $\E_1 \lor \E$ are in $\J^0_{k+1}$;
\item
If $E$ is in $\J^0_{k+1}$ and $D_0$ is quantifier-free, then $D_0\to E$ is in $\J^0_{k+1}$.
\end{itemize}
On the other hand, an intended class $\R_{k+1}^0$ of formulas which are intuitionistically equivalent to some $\Pi_{k+1}^+$-formulas only with assuming the decidability of quantifier-free formulas is desired to satisfy the following clauses:
\begin{itemize}
\item
Formulas which are already ensured to be equivalent to some formulas in $\Sigma_k^+ \cup \Pi_k^+$ are in $\R^0_{k+1}$.
\item
If $U$ and $U'$ are in $\R^0_{k+1}$, then $\forall x U$ and $U\land U'$ are in $\R^0_{k+1}$;
\item
If $U$ is in $\R^0_{k+1}$ and $D_0$ is quantifier-free, then $U \lor D_0$ and $D_0 \lor U$ are in $\R^0_{k+1}$;
\item
If $U$ is in $\R^0_{k+1}$ and $E_1$ is in $\Sigma_1^+$, then $E_1\to U$ is in $\R^0_{k+1}$.
\end{itemize}
In order to relativize them to the semi-classical cases with assuming the decidability of formulas in $\E_n^+\cup \U_n^+$, we consider three cases, namely, the case of  $n>k$, the case of $n=k$, and the case of $n<k$.
In fact, the constructions in the first case are given just by imitating the classical case since the decidability of formulas in $\E_n^+ \cup \U_n^+$ should be sufficient for obtaining equivalent prenex formulas in $\Sigma_k^+ \cup \Pi_k^+$.
The constructions in the second case are given by generalizing the constructions of $\J^0_1$ and $\R^0_1$ for the decidability of quantifier-free formulas to those for formulas in $\Sigma_n^+ \cup \Pi_n^+$.
The constructions in the third case are given by relativizing the constructions of $\J^0_{k+1}$ and $\R^0_{k+1}$ in the above to the situation with assuming the decidability of formulas which can be transformed into formulas in  $\Sigma_n^+ \cup \Pi_n^+$.
Then we define our classes $\J^n_k$ and $\R^n_k$ as follows:

\newpage
\begin{defn}
\label{def: iEiU}
Define $\D^{n}_0 := \J^{n}_0 :=\R^{n}_0:=\F_{0} $ (set of all quantifier-free formulas).
Let $\J^{n}_{k}$ and $\R^{n}_{k}$ have  been already defined.
Put  $\D^{n}_{k}:= \J^{n}_{k} \cup \R^{n}_{k}$.
\begin{enumerate}
\item 
Case of $n>k$:
Classes $\J^{n}_{k+1}$ and $\R^{n}_{k+1}$ are inductively generated by
\begin{enumerate}
        \item
    $D, E\land E', E \lor E', U \to E, \exists x E\in \J^{n}_{k+1};$
        \item
    $D, U\land U', U \lor U',  E \to U,  \forall x U \in \R^{n}_{k+1};$
\end{enumerate}
where $D$ ranges over formulas in $\D^{n}_{k}$, $E$ and $E'$ over those in $\J^{n}_{k+1}$, and $U$ and $U'$ over those in $\R^{n}_{k+1}$ respectively.
\item 
Case of $n=k$:
Classes $\J^{n}_{k+1}$ and $\R^{n}_{k+1}$ are inductively generated by
\begin{enumerate}
        \item
    $D, E\land E', E \lor E', D \to E, \exists x E\in \J^{n}_{k+1};$
        \item
    $D, U\land U', U \lor D, D \lor U,  E \to U,  \forall x U \in \R^{n}_{k+1};$
\end{enumerate}
where $D$ ranges over formulas in $\D^{n}_{k}$, $E$ and $E'$ over those in $\J^{n}_{k+1}$, and $U$ and $U'$ over those in $\R^{n}_{k+1}$ respectively.

\item 
Case of $n<k$:
Classes $\J^{n}_{k+1}$ and $\R^{n}_{k+1}$ are inductively generated by
\begin{enumerate}
        \item
    $D, E\land E', E \lor E_1, E_1 \lor E, D_{0} \to E, \exists x E\in \J^{n}_{k+1};$
        \item
    $D, U\land U', U \lor D_{0}, D_{0}\lor U,  E_{1} \to U,  \forall x U \in \R^{n}_{k+1};$
\end{enumerate}
where $D$ ranges over formulas in $\D^{n}_{k}$, $E$ and $E'$ over those in $\J^{n}_{k+1}$, $U$ and $U'$ over those in $\R^{n}_{k+1}$, $D_{0}$ over those in $\D^{n}_{n}$, and $E_{1}$ over those in $\J^{n}_{n+1}$ respectively.
\end{enumerate}
\end{defn}

Recall that our hierarchical classes $\J^n_{k+1}$ and $\R^n_{k+1}$ are intended to be the classes of formulas which are intuitionistically equivalent to some $\Sigma_{k+1}^+$-formulas and $\Pi_{k+1}^+$-formulas respectively with assuming the decidability of formulas in  $\E_n^+ \cup \U_n^+$.

\begin{remark}
\label{rem: Ekn and Ukn are subclasses of Ek+1n and Uk+1n}
It is straightforward to see $\J^n_k \cup \R^n_k \subset \J^n_{k+1} \cap \R^n_{k+1} $.
In particular, $\D^n_k \subset D^n_{k+k'}$.
\end{remark}

\begin{remark}
    \label{rem: Dnk is smaller than Fk}
Our class $\D^n_k$ is smaller than $\F^+_k$ even if $n \geq k$ (cf. Propositions \ref{prop: k>=n => Ek+=Enk and Uk+=Unk} and \ref{prop: Simga_k subset iE_k and Pi_k subset iU_k} below).
It seems to be crucial to define  $\D^{n}_{k}$ as $\J^{n}_{k} \cup \R^{n}_{k}$ in Definition \ref{def: iEiU} for our proofs in Section \ref{sec: Justification}.
\end{remark}

In the following, we show some basic properties which our classes have.

\begin{lem}\label{lem: subformula property of Dnk}
    For every formula $\vp$ and its subformula $\psi$, if $\vp\in \D^n_k$, then $\psi\in \D^n_k$.
\end{lem}
\begin{proof}
Fix $n$.
By course-of-value induction on $k$, we show that if $\vp\in \D^n_k$, then $\psi \in \D^n_k$ for any subformula $\psi$ of $\vp$.
If $\vp\in D^n_0$, then $\vp$ is quantifier-free and so are its subformulas, and hence, we are done.
For the induction step, assume that the assertion holds up to $k$ and show the assertion for $k+1$ by induction on the structure of $\vp$.
For a prime $\vp$, $\vp$ is only the subformula of $\vp$, and hence, the assertion holds trivially.
Suppose that the assertion holds for $\vp_1$ and $\vp_2$.

Let $\vp \equiv \vp_1 \land \vp_2$.
Suppose $\vp \in \D^n_{k+1}$.
Then $\vp \in \J^n_{k+1}$ or $\vp \in \R^n_{k+1}$.
If $\vp\in \J^n_{k+1}$, then $\vp_1 \land \vp_2 \in \D^n_k$ or $\vp_1, \vp_2 \in \J^n_{k+1}$.
In the former case, we are done by the induction hypothesis for $k$.
We reason in the latter case.
Let $\psi $ be a subformula of $\vp$.
We may assume $\psi \not\equiv \vp$.
Then $\psi $ is a subformula of $\vp_1$ or that of $\vp_2$.
In any case, by the induction hypothesis for $\vp_1$ or $\vp_2$, we have $\psi \in \D^n_{k+1}$.
The case of $\vp \in \R^n_{k+1}$ is similar.

Let $\vp \equiv \vp_1 \lor \vp_2$.
Suppose $\vp \in \D^n_{k+1}$.
Then $\vp \in \J^n_{k+1}$ or $\vp \in \R^n_{k+1}$.

\medskip
\noindent
{\bf Case of $\vp\in \J^n_{k+1}$:}
If $k\leq n$, then $\vp_1 \lor \vp_2 \in \D^n_k$ or $\vp_1, \vp_2 \in \J^n_{k+1}$.
As above, we have $\psi \in \D^n_{k+1}$ for any subformula $\psi$ of $\vp$.
If $k>n$, then (i) $\vp_1 \lor \vp_2 \in \D^n_k$, (ii) $\vp_1\in \J^n_{k+1}$ and $\vp_2 \in \J^n_{n+1}$, or (iii)  $\vp_1\in \J^n_{n+1}$ and $\vp_2 \in \J^n_{k+1}$.

\noindent
Case of (i):
By the induction hypothesis for $k$, we have $\psi \in \D^n_{k}\subset \D^n_{k+1}$ for any subformula $\psi$ of $\vp$.

\noindent
Case of (ii):
Let $\psi $ be a subformula of $\vp$.
We may assume $\psi \not\equiv \vp$.
Then $\psi $ is a subformula of $\vp_1$ or that of $\vp_2$.
In the former case, by the induction hypothesis for $\vp_1$, we have $\psi \in \D^n_{k+1}$.
In the latter case,  by the induction hypothesis for $n+1$ ($\leq k$), we have $\psi \in \D^n_{n+1} \subset \D^n_{k+1}$ (cf. Remark \ref{rem: Ekn and Ukn are subclasses of Ek+1n and Uk+1n}).

\noindent
Case of (iii):
Similar to the case of (ii).

\medskip
\noindent
{\bf Case of $\vp\in \R^n_{k+1}$:}
If $k<n$, then $\vp_1 \lor \vp_2 \in \D^n_k$ or $\vp_1, \vp_2 \in \R^n_{k+1}$.
As above, we have $\psi \in \D^n_{k+1}$ for any subformula $\psi$ of $\vp$.
If $k\geq n$, then (i) $\vp_1 \lor \vp_2 \in \D^n_k$, (ii) $\vp_1\in \R^n_{k+1}$ and $\vp_2 \in \D^n_n$, or (iii)  $\vp_1\in \D^n_n$ and $\vp_2 \in \R^n_{k+1}$.

\noindent
Case of (i):
By the induction hypothesis for $k$, we have $\psi \in \D^n_k \subset \D^n_{k+1}$ for any subformula $\psi$ of $\vp$.

\noindent
Case of (ii):
Let $\psi $ be a subformula of $\vp$.
We may assume $\psi \not\equiv \vp$.
Then $\psi $ is a subformula of $\vp_1$ or that of $\vp_2$.
In the former case, by the induction hypothesis for $\vp_1$, we have $\psi \in \D^n_{k+1}$.
In the latter case,  by the induction hypothesis for $n$ ($\leq k$), we have $\psi \in \D^n_{n} \subset \D^n_{k+1}$ (cf. Remark \ref{rem: Ekn and Ukn are subclasses of Ek+1n and Uk+1n}).

\noindent
Case of (iii):
Similar to the case of (ii).

Let $\vp \equiv \vp_1 \to \vp_2$.
Suppose $\vp \in \D^n_{k+1}$.
Then $\vp \in \J^n_{k+1}$ or $\vp \in \R^n_{k+1}$.

\medskip
\noindent
{\bf Case of $\vp\in \J^n_{k+1}$:}
If $k< n$, then $\vp_1 \to \vp_2 \in \D^n_k$, or $\vp_1 \in \R^n_{k+1}$ and $\vp_2 \in \J^n_{k+1}$.
If $k\geq n$, then $\vp_1 \to \vp_2 \in \D^n_k$, or $\vp_1 \in \D^n_n$ and $\vp_2 \in \J^n_{k+1}$.
As in the previous arguments, in any case, we have $\psi \in \D^n_{k+1}$ for any subformula $\psi$ of $\vp$.

\medskip
\noindent
{\bf Case of $\vp\in \R^n_{k+1}$:}
If $k\leq  n$, then $\vp_1 \to \vp_2 \in \D^n_k$, or $\vp_1 \in \J^n_{k+1}$ and $\vp_2 \in \R^n_{k+1}$.
If $k>n$, then $\vp_1 \to \vp_2 \in \D^n_k$, or $\vp_1 \in \J^n_{n+1}$ and $\vp_2 \in \R^n_{k+1}$.
As in the previous arguments, in any case, we have $\psi \in \D^n_{k+1}$ for any subformula $\psi$ of $\vp$.

Let $\vp \equiv \exists x \vp_1 $.
Suppose $\vp \in \D^n_{k+1}$.
Then $\vp \in \J^n_{k+1}$ or $\vp \in \R^n_{k+1}$.
If $\vp\in \J^n_{k+1}$, then $\exists x \vp_1 \in \D^n_k$ or $\vp_1 \in \J^n_{k+1}$.
As in the previous arguments, in any case, we have $\psi \in \D^n_{k+1}$ for any subformula $\psi$ of $\vp$.
If $\vp\in \R^n_{k+1}$, then $\exists x \vp_1 \in \D^n_k$.
By the induction hypothesis for $k$, we have $\psi \in \D^n_k \subset \D^n_{k+1}$ for any subformula $\psi$ of $\vp$.

The case for $\vp \equiv \forall x \vp_1 $ is verified similarly as in the case for $\vp \equiv \exists x \vp_1 $.    
\end{proof}

\begin{lem}
\label{lem: land Enk and Unk}
    \begin{enumerate}
        \item 
        \label{item: and Enk}
        $\vp \land \psi \in \J^n_{k+1} \Rightarrow \vp, \psi \in \J^n_{k+1}$.
        \item
        \label{item: and Unk}
        $\vp \land \psi \in \R^n_{k+1} \Rightarrow \vp, \psi \in \R^n_{k+1}$.
    \end{enumerate}
\end{lem}
\begin{proof}
    Fix $n$.
    We show \eqref{item: and Enk} and \eqref{item: and Unk} simultaneously by induction on $k$.

First, we show the base step.
Suppose  $\vp \land \psi \in \J^n_1 $.
By the construction of the class $\J^n_1$, we have that $\vp \land \psi \in \D^n_0$ or $\vp, \psi \in \J^n_1$.
In the latter case, we are done.
In the former case, by Lemma \ref{lem: subformula property of Dnk}, we have $\vp, \psi \in \D^n_0$, and hence, $\vp, \psi \in \J^n_1$.
In the same manner, one can also show that if  $\vp \land \psi \in \R^n_1 $, then $\vp, \psi \in \R^n_1$.

Next, we show the induction step.
Assume $k>0$ and \eqref{item: and Enk} and \eqref{item: and Unk}  hold for $k-1$.
Suppose $\vp \land \psi \in \J^n_{k+1} $.
By the construction of $\J^n_{k+1}$, we have $\vp \land \psi \in \D^n_k$ or $\vp, \psi \in \J^n_{k+1}$.
In the latter case, we are done.
In the former case, by Lemma \ref{lem: subformula property of Dnk}, we have $\vp, \psi\in \D^n_k$, and hence, $\vp, \psi\in \J^n_{k+1}$.
In the same manner, one can also show that if  $\vp \land \psi \in \R^n_{k+1} $, then $\vp, \psi \in \R^n_{k+1}$.
\end{proof}

The following two technical lemmas play a crucial role in the proofs of Lemmas \ref{lem: Enk+1 is closed under qn-rule} and \ref{lem: Unk+1 is closed under qn-rule}.
They correspond to the constructions of our classes $\J^n_k$ and $\R^n_k$ with slightly loosing the requirements in the case of $k>n$.
\begin{lem}
\label{lem: lor Enk and Unk}
    \begin{enumerate}
        \item 
        \label{item: or Enk}
         If $\vp \lor \psi \in \J^n_{k+1}$, then
\[
\left\{
\begin{array}{ll}
    \vp, \psi \in \J^n_{k+1} & \text{if }k\leq n,\\[5pt]
    \left(\vp \in \J^n_{k+1} \text{ and } \psi \in \J^n_{n+1} \right) \text{ or }  \left(\vp \in \J^n_{n+1} \text{ and } \psi \in \J^n_{k+1} \right)   & \text{if }k > n.
    \end{array}
\right.
\]
        \item
        \label{item: or Unk}
               If $\vp \lor \psi \in \R^n_{k+1}$, then
\[
\left\{
\begin{array}{ll}
    \vp, \psi \in \R^n_{k+1} & \text{if }k< n,\\[5pt]
    \left(\vp \in \R^n_{k+1} \text{ and } \psi \in \D^n_k \right) \text{ or }  \left(\vp \in \D^n_k \text{ and } \psi \in \R^n_{k+1} \right)   & \text{if }k = n,\\[5pt]
    \left(\vp \in \R^n_{k+1} \text{ and } \psi \in \J^n_{n+1} \right) \text{ or }  \left(\vp \in \J^n_{n+1} \text{ and } \psi \in \R^n_{k+1} \right)   & \text{if }k > n.
    \end{array}
\right.
\]
    \end{enumerate}
\end{lem}
\begin{proof}
 Fix $n$.
    We show \eqref{item: or Enk} and \eqref{item: or Unk} simultaneously by induction on $k$.

First, we show the base step.
To show \eqref{item: or Enk} for $0$, let $\vp \lor \psi \in \J_1^n$.
Note $0\leq n$.
By the construction of the class $\J_1^n$, we have that $\vp \lor \psi \in \D_0^n$ or $\vp, \psi \in \J_1^n$.
In the latter case, we are done.
In the former case, by Lemma \ref{lem: subformula property of Dnk}, we have $\vp, \psi \in \D_0^n$, and hence, $\vp, \psi \in \J_1^n$.
To show \eqref{item: or Unk} for $0$, let $\vp \lor \psi \in \R_1^n$.
If $n>0$, then we have $\vp, \psi \in \R_1^n$ as in the case of \eqref{item: or Enk}.
Assume $n=0$.
By the construction of the class $\R^n_1$, (i) $\vp \lor \psi \in \D_0^n $, (ii) $\vp \in \R_1^n$ and $\psi \in \D_0^n$, or (iii) $\vp \in D_0^n$ and $\psi \in \R_1^n$.
In the second and third cases, we are done.
In the first case, by Lemma \ref{lem: subformula property of Dnk}, we have $\vp, \psi \in \D_0^n$, and hence, $\vp\in \R_1^n$ and  $\psi \in \D_0^n$.

Next, we show the induction step.
Assume $k>0$ and that \eqref{item: or Enk} and \eqref{item: or Unk} hold for $k-1$.
To show $\eqref{item: or Enk} $ for $k$, let $\vp \lor \psi \in \J_{k+1}^n$.

\medskip
\noindent
{\bf Case of $k\leq n$:}
By the construction of the class $\J_{k+1}^n$, $\vp\lor \psi \in \D_k^n$ or $\vp, \psi \in \J_{k+1}^n$.
In the latter case, we are done.
In the former case, by Lemma \ref{lem: subformula property of Dnk}, we have $\vp, \psi \in \D_k^n$, and hence, $\vp, \psi \in \J_{k+1}^n$.

\medskip
\noindent
{\bf Case of $k> n$:}
By the construction of the class $\J_{k+1}^n$, $\vp\lor \psi \in \D_k^n$, $\vp \in \J_{k+1}^n$ and $ \psi \in \J_{n+1}^n$, or $ \vp \in \J_{n+1}^n$ and $\psi \in \J_{k+1}^n$.
In the last two cases, we are done.
We reason in the first case, namely, the case of $\vp\lor \psi \in \D_k^n$.
Suppose $\vp\lor \psi \in \J_k^n$.
If $k-1 \leq n$ (namely, $k-1=n$), by the induction hypothesis, we have $\vp,\psi \in \J_k^n$, and hence, $\vp \in \J_{k+1}^n$ and $\psi \in \J_{n+1}^n$.
If $k-1  > n$, by the induction hypothesis, $\vp \in \J_k^n$ and $\psi \in \J_{n+1}^n$, or $\vp \in \J_{n+1}^n$ and $\psi \in \J_k^n$.
Since $\J_k^n \subset \J_{k+1}^n$, we have that $\vp \in \J_{k+1}^n$ and $\psi \in \J_{n+1}^n$, or $\vp \in \J_{n+1}^n$ and $\psi \in \J_{k+1}^n$.
Next, suppose $\vp\lor \psi \in \R_k^n$.
If  $k-1=n$, by the induction hypothesis, $\vp \in \R_k^n$ and $\psi \in \D_{k-1}^n$, or $\vp \in \D_{k-1}^n$ and $\psi \in \R_k^n$.
Since $\R_k^n \subset \J_{k+1}^n$ and $\D_{k-1}^n=\D_n^n \subset \J_{n+1}^n$, we have that $\vp \in \J_{k+1}^n$ and $\psi \in \J_{n+1}^n$, or $\vp \in \J_{n+1}^n$ and $\psi \in \J_{k+1}^n$.
If  $k-1>n$, by the induction hypothesis, $\vp \in \R_k^n$ and $\psi \in \J_{n+1}^n$, or $\vp \in \J_{n+1}^n$ and $\psi \in \R_k^n$.
Since $\R_k^n \subset \J_{k+1}^n$, again we have that $\vp \in \J_{k+1}^n$ and $\psi \in \J_{n+1}^n$, or $\vp \in \J_{n+1}^n$ and $\psi \in \J_{k+1}^n$.

To show $\eqref{item: or Unk} $ for $k$, let $\vp \lor \psi \in \R_{k+1}^n$.

\medskip
\noindent
{\bf Case of $k< n$:}
By the construction of the class $\R_{k+1}^n$, $\vp\lor \psi \in \D_k^n$ or $\vp, \psi \in \R_{k+1}^n$.
In the latter case, we are done.
In the former case, by Lemma \ref{lem: subformula property of Dnk}, we have $\vp, \psi \in \D_k^n$, and hence, $\vp, \psi \in \R_{k+1}^n$.

\medskip
\noindent
{\bf Case of $k = n$:}
By the construction of the class $\R_{k+1}^n$, $\vp\lor \psi \in \D_k^n$, $\vp \in \R_{k+1}^n$ and $ \psi \in \D_k^n$, or $ \vp \in \D_k^n$ and $\psi \in \R_{k+1}^n$.
In the last two cases, we are done.
In the first case, by Lemma \ref{lem: subformula property of Dnk}, we have $\vp, \psi \in \D_k^n$, and hence, $\vp \in \R_{k+1}^n$ and $ \psi \in \D_k^n$.

\medskip
\noindent
{\bf Case of $k > n$:}
By the construction of the class $\R_{k+1}^n$, $\vp\lor \psi \in \D_k^n$, $\vp \in \R_{k+1}^n$ and $ \psi \in \D_n^n$, or $ \vp \in \D_n^n$ and $\psi \in \R_{k+1}^n$.
In the last two cases, since $\D_n^n\subset \J_{n+1}^n$, we are done.
We reason in the first case, namely, the case of $\vp\lor \psi \in \D_k^n$.
Suppose $\vp\lor \psi \in \J_k^n$.
If $k-1=n$, by the induction hypothesis, we have $\vp,\psi \in \J_k^n$, and hence, $\vp \in \R_{k+1}^n$ and $\psi \in \J_{n+1}^n$.
If $k-1  > n$, by the induction hypothesis, $\vp \in \J_k^n$ and $\psi \in \J_{n+1}^n$, or $\vp \in \J_{n+1}^n$ and $\psi \in \J_k^n$.
Since $\J_k^n\subset \R_{k+1}^n$, we have that $\vp \in \R_{k+1}^n$ and $\psi \in \J_{n+1}^n$, or $\vp \in \J_{n+1}^n$ and $\psi \in \R_{k+1}^n$.
Suppose $\vp\lor \psi \in \R_k^n$.
If $k-1=n$, by the induction hypothesis, $\vp\in \R_k^n$ and $\psi \in \D_n^n$, or $\vp \in \D_n^n$ and $\psi \in \R_k^n$.
Since $\R_k^n \subset \R_{k+1}^n$ and $\D_n^n \subset \J_{n+1}^n$, we have that $\vp \in \R_{k+1}^n$ and $\psi \in \J_{n+1}^n$, or $\vp \in \J_{n+1}^n$ and $\psi \in \R_{k+1}^n$.
If $k-1  > n$, by the induction hypothesis, we are done.
\end{proof}

\begin{lem}
\label{lem: to Enk and Unk}
    \begin{enumerate}
        \item 
        \label{item: to Enk}
         If $\vp \to \psi \in \J^n_{k+1}$, then
\[
\left\{
\begin{array}{ll}
    \vp\in \R^n_{k+1} \text{ and } \psi \in \J^n_{k+1} & \text{if }k< n,\\[5pt]
    \vp \in \D^n_k \text{ and } \psi \in \J^n_{k+1}   & \text{if }k = n,\\[5pt]
    \vp \in \J^n_{n+1} \text{ and  } \psi \in \J^n_{k+1} & \text{if }k > n.
    \end{array}
\right.
\]
        \item
        \label{item: to Unk}
               If $\vp \to \psi \in \R^n_{k+1}$, then
\[
\left\{
\begin{array}{ll}
    \vp\in \J^n_{k+1} \text{ and } \psi \in \R^n_{k+1} & \text{if }k\leq n,\\[5pt]
    \vp \in \J^n_{n+1} \text{ and } \psi \in \R^n_{k+1} & \text{if }k > n.
    \end{array}
\right.
\]
    \end{enumerate}
\end{lem}
\begin{proof}
 Fix $n$.
    We show \eqref{item: to Enk} and \eqref{item: to Unk} simultaneously by induction on $k$.

First, we show the base step.
To show \eqref{item: to Enk} for $0$, let $\vp \to \psi \in \J_1^n$.
If $n>0$, by the construction of the class $\J_1^n$, we have that $\vp \to \psi \in \D_0^n$ or $\vp\in \R_1^n$ and  $\psi \in \J_1^n$.
In the latter case, we are done.
In the former case, by Lemma \ref{lem: subformula property of Dnk}, we have $\vp, \psi \in \D_0^n$, and hence, $\vp\in \R_1^n$ and $\psi \in \J_1^n$.
Assume $n=0$.
By the construction of the class $\J^n_1$, $\vp \to \psi \in \D_0^n $, or $\vp \in \D_0^n$ and $\psi \in \J_1^n$.
In the latter case, we are done.
In the former case, by Lemma \ref{lem: subformula property of Dnk}, we have $\vp, \psi \in \D_0^n$, and hence, $\vp\in \D_0^n$ and $\psi \in \J_1^n$.
To show \eqref{item: to Unk} for $0$, let $\vp \to \psi \in \R_1^n$.
Note $0\leq n$.
By the construction of the class $\R_1^n$, $\vp \to \psi \in \D_0^n $, or $\vp \in \J_1^n$ and $\psi \in \R_1^n$.
In the latter case, we are done.
In the former case, by Lemma \ref{lem: subformula property of Dnk}, we have $\vp\in \J_1^n$ and $\psi \in \R_1^n$  as above.

Next, we show the induction step.
Assume $k>0$ and that \eqref{item: to Enk} and \eqref{item: to Unk} hold for $k-1$.
To show $\eqref{item: to Enk} $ for $k$, let $\vp \to \psi \in \J_{k+1}^n$.

\medskip
\noindent
{\bf Case of $k<n$:}
By the construction of the class $\J_{k+1}^n$, $\vp\to \psi \in \D_k^n$, or $\vp\in \R_{k+1}^n $ and $\psi \in \J_{k+1}^n$.
In the latter case, we are done.
In the former case, by Lemma \ref{lem: subformula property of Dnk}, we have $\vp, \psi \in \D_k^n$, and hence, $\vp\in \R_{k+1}^n$ and  $\psi \in \J_{k+1}^n$.

\medskip
\noindent
{\bf Case of $k=n$:}
By the construction of the class $\J_{k+1}^n$, $\vp\to \psi \in \D_k^n$, or $\vp\in \D_k^n $ and $\psi \in \J_{k+1}^n$.
In the latter case, we are done.
In the former case, by Lemma \ref{lem: subformula property of Dnk}, we have $\vp, \psi \in \D_k^n$, and hence, $\vp\in \D_k^n$ and  $\psi \in \J_{k+1}^n$.

\medskip
\noindent
{\bf Case of $k> n$  (namely, $k-1\geq n$):}
By the construction of the class $\J_{k+1}^n$, $\vp\to \psi \in \D_k^n$, or $\vp\in \D_n^n $ and $\psi \in \J_{k+1}^n$.
In the latter case, since $\D_n^n \subset \J_{n+1}^n$ we are done.
We reason in the former case, namely, the case of $\vp\to \psi \in \D_k^n$.
Suppose $ \vp\to \psi \in \J_k^n$.
If $k-1 =n$, by the induction hypothesis, $\vp\in \D_n^n$ and $\psi \in \J_k^n$.
Since $ \D_n^n \subset  \J_{n+1}^n$ and $\J_k^n \subset \J_{k+1}^n$, we have $\vp\in \J_{n+1}^n$ and $\psi \in \J_{k+1}^n$.
If $k-1 >n$, by the induction hypothesis, we are done.
Suppose $ \vp\to \psi \in \R_k^n$.
If $k-1 =n$, by the induction hypothesis, $\vp\in \J_k^n$ and $\psi \in \R_k^n$.
Since $\J_k^n=\J_{n+1}^n$ and $\R_k^n \subset \J_{k+1}^n$, we have that $\vp \in \J_{n+1}^n$ and $\psi \in \J_{k+1}^n$.
If $k-1 >n$, by the induction hypothesis, we have that $\vp \in \J_{n+1}^n$ and $\psi \in \R_{k}^n$, and hence, $\vp \in \J_{n+1}^n$ and $\psi \in \J_{k+1}^n$.

To show $\eqref{item: to Unk} $ for $k$, let $\vp \to \psi \in \R_{k+1}^n$.

\medskip
\noindent
{\bf Case of $k\leq n$:}
By the construction of the class $\R_{k+1}^n$, $\vp\to \psi \in \D_k^n$, or $\vp\in \J_{k+1}^n $ and $\psi \in \R_{k+1}^n$.
In the latter case, we are done.
In the former case, by Lemma \ref{lem: subformula property of Dnk}, we have $\vp, \psi \in \D_k^n$, and hence, $\vp\in \J_{k+1}^n$ and  $\psi \in \R_{k+1}^n$.

\medskip
\noindent
{\bf Case of $k>n$ (namely, $k-1 \geq n$):}
By the construction of the class $\R_{k+1}^n$, $\vp\to \psi \in \D_k^n$, or $\vp\in \J_{n+1}^n $ and $\psi \in \R_{k+1}^n$.
    In the latter case, we are done.
    We reason in the former case, namely, the case of $\vp\to \psi \in \D_k^n$.
    Suppose $ \vp\to \psi \in \J_k^n$.
Since $k-1\geq n$, by the induction hypothesis, we have that $\vp \in \J_{n+1}^n $ and $\psi \in \J_{k}^n$, and hence, $\vp \in \J_{n+1}^n $ and $\psi \in \R_{k+1}^n$
Suppose $ \vp\to \psi \in \R_k^n$.
If $k-1 =n$, by the induction hypothesis, $\vp\in \J_k^n$ and $\psi \in \R_k^n$.
Since $\J_k^n=\J_{n+1}^n$ and $\R_k^n \subset \R_{k+1}^n$, we have that $\vp \in \J_{n+1}^n$ and $\psi \in \R_{k+1}^n$.
If $k-1 >n$, by the induction hypothesis, we are done.
\end{proof}

\begin{lem}
\label{lem: exists Enk and Unk}
    \begin{enumerate}
        \item 
        \label{item: exists Enk}
         If $\exists x \vp  \in \J^n_{k+1}$, then $\vp  \in \J^n_{k+1}$.
        \item
        \label{item: exists Unk}
         If $\exists x \vp  \in \R^n_{k+1}$, then $\exists x \vp  \in \J^n_k$ and $k>0$.
    \end{enumerate}
\end{lem}
\begin{proof}
\eqref{item: exists Enk}:
Let $\exists x \vp \in \J_{k+1}^n$.
By the construction of $\J_{k+1}^n$, $\exists x \vp \in \D_k^n$ or $\vp \in \J_{k+1}^n$.
In the latter case, we are done.
In the former case, by Lemma \ref{lem: subformula property of Dnk}, we have $\vp\in \D_k^n$, and hence, $\vp\in \J_{k+1}^n$

\eqref{item: exists Unk}:
Let $\exists x \vp \in \R_{k+1}^n$.
By the construction of $\R_{k+1}^n$, we have $\exists x \vp \in \D_k^n$.
Since $\D_0^n $ is the set of quantifier-free formulas, we have $k>0$.
If $\exists x \vp \in \R_k^n$, by  the construction of $\R_k^n$, we have $\exists x \vp \in \D_{k-1}^n$, and hence, $\exists x \vp \in \J_k^n$.
Thus we have  $\exists x \vp \in \J_k^n$ by $\exists x \vp \in \D_k^n$.
\end{proof}

\begin{lem}
\label{lem: forall Enk and Unk}
    \begin{enumerate}
        \item 
        \label{item: forall Enk}
         If $\forall x \vp  \in \J^n_{k+1}$, then $\forall x \vp  \in \R^n_k$ and $k>0$.
        \item
        \label{item: forall Unk}
         If $\forall x \vp  \in \R^n_{k+1}$, then $\vp  \in \R^n_{k+1}$.
    \end{enumerate}
\end{lem}
\begin{proof}
    Similar to the proof of Lemma \ref{lem: exists Enk and Unk}.
\end{proof}

 Recall that the classes $\mathcal{J}_{k+1}$ and $\mathcal{R}_{k+1}$ in \cite[Definition 3.11]{FK23} are defined simultaneously as follows:
\begin{enumerate}
        \item
    $F, E\land E', E \lor E', U \to E, \exists x E\in \mathcal{J}_{k+1};$
        \item
    $F, U\land U', U \lor U', E \to U,  \forall x U \in \mathcal{R}_{k+1};$
\end{enumerate}
where $F$ ranges over formulas in $\F_k^+$, $E$ and $E'$ over those in $\mathcal{J}_{k+1}$, and $U$ and $U'$ over those in $\mathcal{R}_{k+1}$ respectively.
In \cite[Proposition 4.1]{FK23}, the authors have shown that   $\E_{k}^+ = \mathcal{J}_{k}$ and $\U_{k}^+ = \mathcal{R}_{k}$.
The following proposition states that this is also the case for $\J_{k}^{n}$ and $\R_k^n$ respectively for all $k$ and $n$ such that $k\leq n$.

 \begin{prop}
 \label{prop: k>=n => Ek+=Enk and Uk+=Unk}
For all $n\in \N$ and $k\in\N$ such that $k\leq n$, $\E^+_k = \J^n_k$ and $\U^+_k = \R^n_k$.
 \end{prop}
 \begin{proof}
Fix $n\in \N$.
By induction on $k$, we show $\E^+_k = \J^n_k$ and $\U^+_k = \R^n_k$ for all $k\leq n$.
The base step is trivial.
For the induction step, assume $k+1 \leq n$ (then $k<n$), $\E^+_k = \J^n_k$ and $\U^+_k = \R^n_k$.

In the following, we show that for any formula $\vp$, $\vp \in \E^+_{k+1} \Leftrightarrow \vp \in \J^n_{k+1}$ and  $\vp \in \U^+_{k+1} \Leftrightarrow \vp \in \R^n_{k+1}$, by induction on the structure of formulas.

For a prime formula $\vp$, the assertion holds since $\vp \in \E^+_0 =\U^+_0 = \J^n_0 =\R^n_0$.
For the induction step, assume that the assertion holds for $\vp_1$ and $\vp_2$.

Let $\vp \equiv \vp_1 \land \vp_2$.
By \cite[Lemma 2]{FK24}, we have that $\vp_1 \land \vp_2 \in \E^+_{k+1} $ if and only if $\vp_1, \vp_2 \in \E^+_{k+1}$, which is equivalent to $\vp_1, \vp_2 \in \J^n_{k+1}$ by the induction hypothesis for $\vp_1$ and $\vp_2$.
Now, $\vp_1, \vp_2 \in \J^n_{k+1}$ implies $\vp_1 \land \vp_2 \in \J^n_{k+1}$.
On the other hand, if $\vp_1 \land \vp_2 \in \J^n_{k+1} $, by Lemma \ref{lem: land Enk and Unk}, we have $\vp_1 , \vp_2 \in \J^n_{k+1}$.
Thus we have shown $\vp_1 \land \vp_2 \in \E^+_{k+1}  \Leftrightarrow \vp_1 \land \vp_2 \in \J^n_{k+1}$.
In the same manner, we also have  $\vp_1 \land \vp_2 \in \U^+_{k+1}  \Leftrightarrow \vp_1 \land \vp_2 \in \R^n_{k+1}$.

The case of $\vp \equiv \vp_1 \lor \vp_2$ is verified as in the case of $\vp \equiv \vp_1 \land \vp_2$ with using Lemma \ref{lem: lor Enk and Unk} instead of Lemma \ref{lem: land Enk and Unk}, which works since $k<n$ now (cf. the proof of the case of $\vp \equiv \vp_1 \to \vp_2$ below).

Let $\vp \equiv \vp_1 \to \vp_2$.
By \cite[Lemma 2]{FK24}, we have that $\vp_1 \to \vp_2 \in \E^+_{k+1} $ if and only if $\vp_1\in \U^+_{k+1}$ and $\vp_2 \in \E^+_{k+1}$, which is equivalent to that $\vp_1\in \R^n_{k+1}$ and $\vp_2 \in \J^n_{k+1}$ by the induction hypothesis for $\vp_1$ and $\vp_2$.
Since $k< n$ now, $\vp_1\in \R^n_{k+1}$ and $\vp_2 \in \J^n_{k+1}$ imply $\vp_1 \to \vp_2 \in \J^n_{k+1}$.
On the other hand, if  $\vp_1 \to \vp_2 \in \J^n_{k+1}$, by Lemma \ref{lem: to Enk and Unk} (note $k<n$ now), we have $\vp_1 \in \R_{k+1}^n$ and $\vp_2 \in \J_{k+1}^n$.
Thus we have shown $\vp_1 \to \vp_2 \in \E^+_{k+1}  \Leftrightarrow \vp_1 \to \vp_2 \in \J^n_{k+1}$.
In the same manner, we also have  $\vp_1 \to \vp_2 \in \U^+_{k+1}  \Leftrightarrow \vp_1 \to \vp_2 \in \R^n_{k+1}$.

Let $\vp \equiv \exists x \vp_1$.
By \cite[Lemma 2]{FK24}, we have that $\exists x \vp_1  \in \E^+_{k+1} $ if and only if $\vp_1\in \E^+_{k+1}$, which is equivalent to $\vp_1 \in \J^n_{k+1}$ by the induction hypothesis for $\vp_1$.
Now  $\vp_1 \in \J^n_{k+1}$ implies  $\exists x \vp_1 \in \J^n_{k+1}$.
On the other hand, if $\exists x \vp_1 \in \J^n_{k+1}$, by Lemma \ref{lem: exists Enk and Unk}, we have $\vp_1 \in \J_{k+1}^n$.
Thus we have shown $\exists x \vp_1\in \E^+_{k+1}  \Leftrightarrow \exists x \vp_1 \in \J^n_{k+1}$.
Next we show the assertion for $\U^+_{k+1}$ and $\R^n_{k+1}$.
By \cite[Lemma 2]{FK24}, we have that $\exists x \vp_1  \in \U^+_{k+1} $ if and only if $\exists x \vp_1\in \E^+_{k}$, which is equivalent to $\exists x \vp_1 \in \J^n_{k}$ by the induction hypothesis for $k$.
Since $\J^n_{k} \subset \R^n_{k+1}$, if $\exists x \vp_1 \in \J^n_{k}$, then $\exists x \vp_1 \in \R^n_{k+1}$.
On the other hand, if $\exists x \vp_1 \in \R^n_{k+1}$, by Lemma \ref{lem: exists Enk and Unk}, we have  $\exists x \vp_1 \in \J^n_{k}$.
Thus we have shown $\exists x \vp_1\in \U^+_{k+1}  \Leftrightarrow \exists x \vp_1 \in \R^n_{k+1}$.

The case of $\vp \equiv \forall x \vp_1 $ is verified in a similar manner as the case of $\vp \equiv \exists x \vp_1 $ with using Lemma \ref{lem: forall Enk and Unk} instead of Lemma \ref{lem: exists Enk and Unk}.
\end{proof}

\begin{remark}
    \label{rem: classes C+}
In \cite[Section 5]{FK24}, authors defined the classes $\C_{k+1}$ as $\PF_k\cup \E_{k+1} \cup \U_{k+1}$.
Now, let us consider their cumulative variants.
Let $\C^+_0 := \F_0$ and $ \begin{displaystyle}
    \C^+_{k+1} := \bigcup_{i\leq k} C_{i+1}.
\end{displaystyle}$
Then, by Proposition \ref{prop: k>=n => Ek+=Enk and Uk+=Unk}, we have
\[
\C^+_{k+1}=\PF^+_k \cup \E_{k+1} \cup \U_{k+1} = \E^+_{k+1} \cup \U^+_{k+1} = \J^{k+1}_{k+1} \cup \R^{k+1}_{k+1} = \D^{k+1}_{k+1}.
\]
These cumulative variants are used for the restriction of prenex normalization in Definition \ref{def: rules of qn}.
\end{remark}
 
\begin{prop}
\label{prop: Simga_k subset iE_k and Pi_k subset iU_k}
\begin{enumerate}
\item
\label{item: Sigmak+ is a subset of Ek0}
$\Sigma_{k}^{+} \subset \J^0_{k}$.
\item
\label{item: Pik+ is a subset of Uk0}
 $\Pi_{k}^{+} \subset \R^0_{k}$.
\item
\label{item: Ekn is a subset of Ekn+1}
$\J^n_{k} \subset \J^{n+1}_{k}$.
\item
\label{item: Ukn is a subset of Ukn+1}
$\R^n_{k} \subset \R^{n+1}_{k}$.
 \end{enumerate}
 \end{prop}
 \begin{proof}
     The clauses \eqref{item: Sigmak+ is a subset of Ek0} and \eqref{item: Pik+ is a subset of Uk0} are verified by simultaneous induction on $k$.
    
 For the clauses \eqref{item: Ekn is a subset of Ekn+1} and \eqref{item: Ukn is a subset of Ukn+1}, we show that $\J^n_{k} \subset \J^{n+1}_{k}$ and $\R^n_{k} \subset \R^{n+1}_{k}$ for all $k$ and $n$ by course-of-value induction on $k$.
The base step is trivial.
For the induction step, assume $\J^n_{k'} \subset \J^{n+1}_{k'}$ and $\R^n_{k'} \subset \R^{n+1}_{k'}$ for all $k' \leq k$ and $n$.

In the following, we fix $n$ and show that for any formula $\vp$, $\vp\in \J_{k+1}^n \Rightarrow \vp\in \J_{k+1}^{n+1}$ and $\vp\in \R_{k+1}^n \Rightarrow \vp\in \R_{k+1}^{n+1}$ by induction on the structure of formulas.
For a prime formula $\vp$, the assertion holds by Remark \ref{rem: Ekn and Ukn are subclasses of Ek+1n and Uk+1n} since $\vp \in \J^{n+1}_0 =\R^{n+1}_0$.
For the induction step, assume that the assertion holds for $\vp_1$ and $\vp_2$.
By Proposition \ref{prop: k>=n => Ek+=Enk and Uk+=Unk}, it suffices to reason only in the case of $k+1>n$, equivalently, $k\geq n$.

Suppose $\vp_1 \land \vp_2 \in \J_{k+1}^n$.
By Lemma \ref{lem: land Enk and Unk}, we have $\vp_1, \vp_2 \in \J_{k+1}^n$.
Then, by the induction hypothesis for $\vp_1$ and $\vp_2$, we have $\vp_1, \vp_2 \in \J_{k+1}^{n+1}$, and hence, $\vp_1 \land \vp_2 \in \J_{k+1}^{n+1}$.

In the same manner, one can show that if $\vp_1 \land \vp_2 \in \R_{k+1}^n$, then $\vp_1 \land \vp_2 \in \R_{k+1}^{n+1}$.

Suppose $ \vp_1 \lor \vp_2 \in \J_{k+1}^n$.
We first reason in the case of $k=n$.
By Lemma \ref{lem: lor Enk and Unk}, we have $\vp_1, \vp_2 \in \J_{k+1}^n$.
Then, by the induction hypothesis for $\vp_1$ and $\vp_2$, we have $\vp_1, \vp_2 \in \J_{k+1}^{n+1}$, and hence, $\vp_1 \lor \vp_2 \in \J_{k+1}^{n+1}$ since $k<n+1$.
We next reason in the case of $k>n$.
Then, by Lemma \ref{lem: lor Enk and Unk}, we have that $\vp_1 \in \J_{k+1}^n$ and $\vp_2 \in \J_{n+1}^n$, or $\vp_1 \in \J_{n+1}^n$ and $\vp_2 \in \J_{k+1}^n$.
Then, by the induction hypothesis for $\vp_1$ and $\vp_2$, and also the induction hypothesis for $n+1 \, (\leq k)$, we have that $\vp_1 \in \J_{k+1}^{n+1}$ and $\vp_2 \in \J_{n+1}^{n+1}$, or $\vp_1 \in \J_{n+1}^{n+1}$ and $\vp_2 \in \J_{k+1}^{n+1}$.
Since $k\geq n+1$, by Remark \ref{rem: Ekn and Ukn are subclasses of Ek+1n and Uk+1n} and the construction of the class $\J_{k+1}^{n+1}$, we have $ \vp_1 \lor \vp_2 \in \J_{k+1}^{n+1}$.

Suppose $ \vp_1 \lor \vp_2 \in \R_{k+1}^n$.

\medskip
\noindent
{\bf Case of $k=n$:}
By Lemma \ref{lem: lor Enk and Unk}, 
we have that  $\vp_1 \in \R_{k+1}^n$ and $\vp_2 \in \D_{k}^n$, or $\vp_1 \in \D_{k}^n$ and $\vp_2 \in \R_{k+1}^n$.
Then, by the induction hypothesis for $\vp_1$ and $\vp_2$, and also the induction hypothesis for $k$, we have that  $\vp_1 \in \R_{k+1}^{n+1}$ and $\vp_2 \in \D_{k}^{n+1}$, or $\vp_1 \in \D_{k}^{n+1}$ and $\vp_2 \in \R_{k+1}^{n+1}$.
Since $k< n+1$ and $\D_{k}^n\subset \R_{k+1}^n$, by Remark \ref{rem: Ekn and Ukn are subclasses of Ek+1n and Uk+1n} and the construction of the class $\R_{k+1}^{n+1}$, we have $ \vp_1 \lor \vp_2 \in \R_{k+1}^{n+1}$.

\medskip
\noindent
{\bf Case of $k>n$:}
By Lemma \ref{lem: lor Enk and Unk}, we have that  $\vp_1 \in \R_{k+1}^n$ and $\vp_2 \in \J_{n+1}^n$, or $\vp_1 \in \J_{n+1}^n$ and $\vp_2 \in \R_{k+1}^n$.
Then, by the induction hypothesis for $\vp_1$ and $\vp_2$, and also the induction hypothesis for $n+1\, (\leq k)$, we have that  $\vp_1 \in \R_{k+1}^{n+1}$ and $\vp_2 \in \J_{n+1}^{n+1}$, or $\vp_1 \in \J_{n+1}^{n+1}$ and $\vp_2 \in \R_{k+1}^{n+1}$.
Since $k\geq  n+1$ and $\J_{n+1}^{n+1}\subset \D_{n+1}^{n+1}$, by the construction of the class $\R_{k+1}^{n+1}$, we have $ \vp_1 \lor \vp_2 \in \R_{k+1}^{n+1}$.

Suppose $ \vp_1 \to \vp_2 \in \J_{k+1}^n$.

\medskip
\noindent
{\bf Case of $k=n$:}
By Lemma \ref{lem: to Enk and Unk},
we have that $\vp_1\in \D_k^n$ and $\vp_2 \in \J_{k+1}^n$.
Then, by the induction hypothesis for $k$ and also the induction hypothesis for $\vp_2$, we have $\vp_1\in \D_k^{n+1} \subset \R_{k+1}^{n+1}$ and $\vp_2 \in \J_{k+1}^{n+1}$, and hence, $\vp_1 \to \vp_2 \in \J_{k+1}^{n+1}$ since $k<n+1$.

\medskip
\noindent
{\bf Case of $k>n$:}
Then, by Lemma \ref{lem: to Enk and Unk}, we have that $\vp_1\in \J_{n+1}^n$ and $\vp_2 \in \J_{k+1}^n$.
Then, by the induction hypothesis for $n+1 \, (\leq k)$ and the induction hypothesis for $\vp_2$, we have $\vp_1\in \J_{n+1}^{n+1} \subset \D_{n+1}^{n+1}$ and $\vp_2 \in \J_{k+1}^{n+1}$, and hence, $\vp_1 \to \vp_2 \in \J_{k+1}^{n+1}$ since $k\geq n+1$.

Suppose $ \vp_1 \to \vp_2 \in \R_{k+1}^n$.

\medskip
\noindent
{\bf Case of $k=n$:}
By Lemma \ref{lem: to Enk and Unk}, we have that  $\vp_1 \in \J_{k+1}^n$ and $\vp_2 \in \R_{k+1}^n$.
Then, by the induction hypothesis for $\vp_1$ and $\vp_2$, we have that  $\vp_1 \in \J_{k+1}^{n+1}$ and $\vp_2 \in \R_{k+1}^{n+1}$.
Since $k< n+1$, by the construction of the class $\R_{k+1}^{n+1}$, we have $ \vp_1 \to \vp_2 \in \R_{k+1}^{n+1}$.

\medskip
\noindent
{\bf Case of $k>n$:}
By Lemma \ref{lem: to Enk and Unk}, we have that  $\vp_1 \in \J_{n+1}^n$ and $\vp_2 \in \R_{k+1}^n$.
Then, by the induction hypothesis for $n+1 \, (\leq k)$ and also the induction hypothesis for $\vp_2$, we have that  $\vp_1 \in \J_{n+1}^{n+1} \subset \J_{n+2}^{n+1}$ and $\vp_2 \in \R_{k+1}^{n+1}$.
Since $k\geq  n+1$, by the construction of the class $\R_{k+1}^{n+1}$, we have $ \vp_1 \to \vp_2 \in \R_{k+1}^{n+1}$.

Suppose $\exists x \vp_1 \in \J_{k+1}^n$.
By Lemma \ref{lem: exists Enk and Unk}, we have $\vp_1 \in \J_{k+1}^n$.
Then, by the induction hypothesis for $\vp_1$, we have $\vp_1 \in \J_{k+1}^{n+1}$, and hence, $\exists x \vp_1 \in \J_{k+1}^{n+1}$.

Suppose $\exists x \vp_1 \in \R_{k+1}^n$.
By Lemma \ref{lem: exists Enk and Unk}, we have $\exists x \vp_1 \in \J_{k}^n$.
Then, by the induction hypothesis for $k$, we have $\exists x \vp_1 \in \J_{k}^{n+1}$, and hence, $\exists x \vp_1 \in \R_{k+1}^{n+1}$ by Remark \ref{rem: Ekn and Ukn are subclasses of Ek+1n and Uk+1n}.

As in the case for $\exists x \vp_1 $, one can show that $\forall x \vp_1 \in \J_{k+1}^n \Rightarrow \forall x \vp_1\in \J_{k+1}^{n+1}$ and $\forall x \vp_1 \in \R_{k+1}^n \Rightarrow \forall x \vp_1\in \R_{k+1}^{n+1}$ by using Lemma \ref{lem: forall Enk and Unk} instead of Lemma \ref{lem: exists Enk and Unk}.
\end{proof}

 Propositions \ref{prop: k>=n => Ek+=Enk and Uk+=Unk} and \ref{prop: Simga_k subset iE_k and Pi_k subset iU_k} and Remark \ref{rem: Ekn and Ukn are subclasses of Ek+1n and Uk+1n} state that our classes $\J_k^n$ and $\R_k^n$ have the following relation:
 \begin{itemize}
 \item
$\J_0^n\subset \J_1^n\subset \dots $;
\item
$\R_0^n\subset \R_1^n\subset \dots $;
\item
$\Sigma_{k}^{+} \subset \J^{0}_{k}\subset \J^{1}_{k}\subset \dots \subset \J^{k}_{k}  = \J^{k+1}_{k} =\dots =   \E_{k}^{+}$;
\item
$\Pi_{k}^{+} \subset \R^{0}_{k}\subset \R^{1}_{k}\subset \dots \subset \R^{k}_{k}  = \R^{k+1}_{k} =\dots =   \U_{k}^{+}$;
\end{itemize}
 which is visualized in Table \ref{table: HC}.


 \begin{table}
\[
\begin{array}{l|ccccccc}
\hline\\[-10pt]
\text{Intuitionistic Hierarchy}&\J^{0}_{0} & \multicolumn{1}{|c}{\J^{0}_{1}} & \J^{0}_{2} & \J^{0}_{3} & \dots & \J^{0}_{k+1}&\dots \\[3pt]
\cline{3-3}\\[-10pt]
&\J^{1}_{0} &  \J^{1}_{1} & \multicolumn{1}{|c}{\J^{1}_{2}} & \J^{1}_{3} & \dots & \J^{1}_{k+1}&\dots \\[3pt]
\cline{4-4}\\[-10pt]
&\J^{2}_{0} &  \J^{2}_{1} & \J^{2}_{2} & \multicolumn{1}{|c}{\J^{2}_{3}} & \dots & \J^{2}_{k+1}&\dots \\[3pt]
\cline{5-5}\\[3pt]
&\vdots &&&&\ddots & \vdots & \\[3pt]
&&&&&&  \multicolumn{1}{|c}{\J^{k}_{k+1}}&\J^{k}_{k+2} \\[4pt]
\cline{7-7}\\[-10pt]
&&&&&&  \J^{k+1}_{k+1}& \multicolumn{1}{|c}{\J^{k+1}_{k+2}} \\[3pt]
\cline{8-8}\\[3pt]
&\parallel&\parallel&\parallel&\parallel&\dots & \parallel&\\[5pt]
\text{Classical Hierarchy}&\E^{+}_{0} & \E^{+}_{1} &\E^{+}_{2} &\E^{+}_{3} & \dots &\E^{+}_{k+1}  &\dots\\[2pt]
\hline
\end{array}
\]
\caption{Hierarchical Classes}
\label{table: HC}
\end{table}

\section{Justification by semi-classical prenex normalization}
\label{sec: Justification}
In what follows, we justify our semi-classical hierarchical classes by showing that $\J^{n}_{k}$ and $\R^{n}_{k}$ are exactly the classes corresponding to $\Sigma_{k}^{+}$ and $\Pi_{k}^{+}$ respectively with respect to the prenex normalization restricted to some reasonable formula classes of degree $n$.

Based on the formulation of the classes in Definition \ref{def: iEiU}, we shall give an appropriate definition of semi-classical transformations which are possible in intuitionistic logic augmented with assuming the decidability of formulas in $\E^+_n\cup \U^+_n$.
Following \cite[Section 5]{FK24}, let $\C^+_n := \E^+_n \cup \U^+_n$ (cf. Remark \ref{rem: classes C+}).
First, since the rules $(\to \exists)$ and the converses of $(\forall \lor)$ and $(\lor \forall)$ are derivable only for decidable $\delta$ (as already mentioned in Section  \ref{sec: scPN}), we should restrict the rules $(\to \exists)$ and the converses of  $(\forall \lor)$ and $(\lor \forall)$ to those with $\delta\in \C^+_n$.
The restriction to the rule $(\forall \to)$ is more delicate.
In fact, $(\forall x \QF{\xi}(x) \to \perp) \to \exists x (\QF{\xi}(x)\to \perp)$ with quantifier-free $\QF{\xi}(x)$ already implies $\Sigma_1$-DNE over $\HA$, and hence, is not provable in $\HA$.
On the other hand, the rule $(\forall \to)$ is derivable in an intuitionistic theory which proves the law-of-excluded-middle for $\exists x \neg \xi (x)$ and the double-negation-elimination for $\xi (x)$ (see the proof of Theorem \ref{thm: equivalents of Sn-LEM over HA} below).
For example, the restricted variant of the rule $(\forall \to)$ where $\xi(x)\in \U_1^+$ is derivable in $\HA + \Sigma_1\text{-}{\rm LEM}$.
Thus the rule $(\forall \to)$ should not be contained even for decidable $\xi(x)$ when $n=0$ but it should be contained for $\xi (x)$ such that $\exists x \neg \xi (x) $ is decidable for $n>0$.

As a justification of our restriction of the semi-classical prenex normalizations, we first present some facts in the context of intuitionistic arithmetic.
\begin{thm}
\label{thm: equivalents of Sn-LEM over HA}
For $n>0$, the following are equivalent over $\HA:$
\begin{enumerate}
\item
\label{item: Sn-LEM}
$\LEM{\Sigma_{n}};$
\item
\label{item: Dnn-LEM}
$\LEM{\C^+_n}$
\item
\label{item: forall -> rule restricted to Unn}
$(\forall x \xi(x) \to \delta) \to \exists x(\xi(x) \to \delta)$ where $\xi(x) \in \U^+_{n} \text{ and }x \notin \FV{\delta};$
\item
\label{item: ->exists rule restricted to Dnn}
$( \delta \to \exists x \xi(x)) \to \exists x( \delta \to \xi(x))$ where $\delta \in \C^+_{n} \text{ and }x \notin \FV{\delta};$
\item
\label{item: forall v rule restricted to Dnn}
$\forall x( \xi(x) \lor \delta) \to \forall x \xi(x) \lor \delta$ where $\delta \in \C^+_{n} \text{ and }x \notin \FV{\delta}.$
\end{enumerate}
\end{thm}
\begin{proof}
Fix $n>0$.
Implication $(\ref{item: Sn-LEM} \Rightarrow \ref{item: Dnn-LEM})$ follows from \cite[Proposition 6.9]{FK23}.

$ (\ref{item: Dnn-LEM} \Rightarrow \ref{item: forall -> rule restricted to Unn}  )$:
We reason in $\HA + \LEM{\C^+_n}$.
Fix $\xi(x) \in  \U^+_{n}$ and $\delta$ such that $x \notin \FV{\delta}$.
Then we have $\neg \xi(x) \equiv \xi(x) \to \bot \in \E^+_n$, and hence, $\exists  x \neg \xi(x) \in \E^+_n$.
By $\LEM{\C^+_n}$, we have $\exists x \neg \xi(x) \lor \neg \exists x \neg \xi(x)$.
In the former case, by the reasoning in $\HA$, we have $\exists x( \xi(x) \to \delta)$ straightforwardly.
In the latter case, by the reasoning in $\HA$, we have $\forall x \neg \neg \xi(x)$, and hence, $\forall x  \xi(x)$ by $\DNE{\U^+_n}$, which is derived from $\LEM{\C^+_n}$.
Therefore, in any case, we have $(\forall x  \xi(x) \to \delta)\to \exists x(\xi(x) \to \delta)$.

$ (\ref{item: forall -> rule restricted to Unn} \Rightarrow \ref{item: Sn-LEM} )$:
We show that for all $i\leq n$, $\HA +\eqref{item: forall -> rule restricted to Unn}$
proves $\LEM{\Sigma_i}$ by induction on $i$.
The base case is trivial.
For the induction step, assume $i+1 \leq n$ and that $\HA +\eqref{item: forall -> rule restricted to Unn}$ proves $\LEM{\Sigma_i}$.
We show $\exists \ol{x}  \vp(\ol{x}) \lor \neg \exists \ol{x} \vp(\ol{x})$ for $ \vp(\ol{x}) \in \Pi_i$ inside $\HA +\eqref{item: forall -> rule restricted to Unn}$.
Since $\exists \ol{x}  \vp(\ol{x}) \lor \neg \exists \ol{x}  \vp(\ol{x})$ is equivalent (over $\HA$) to
\begin{equation}
\label{eq: equivalent of the instance of Sn-LEM}
\exists \ol{x}  \left( \vp(\ol{x} ) \lor \forall \ol{x}  \neg \vp(\ol{x} )\right),\tag{a}
\end{equation}
it suffices to derive \eqref{eq: equivalent of the instance of Sn-LEM}.
By Proposition \ref{prop: Simga_k subset iE_k and Pi_k subset iU_k}, we have $\vp(\ol{x} ) \in \U_i^+$ and $\bot \in \E_i^+$.
Then we have $\neg \vp(\ol{x} ) \in \E_i^+$, and hence, $ \forall \ol{x}  \neg \vp(\ol{x} ) \in \U_{i+1}^+ \subset \U_n^+$.
Therefore, by \eqref{item: forall -> rule restricted to Unn}, we have
\begin{equation}
    \label{eq: formula from (4)}
\exists \ol{x}  \left( \neg \vp(\ol{x} ) \to \forall \ol{x} \neg \vp(\ol{x} )\right).\tag{b}
\end{equation}
Since $\forall \ol{x} \left( \vp(\ol{x} ) \lor \neg \vp(\ol{x} ) \right)$ is derived from $\LEM{\Sigma_i}$, we obtain \eqref{eq: equivalent of the instance of Sn-LEM} from \eqref{eq: formula from (4)}.

$(\ref{item: Dnn-LEM} \Rightarrow \ref{item: ->exists rule restricted to Dnn})$:
We reason in $\HA + \LEM{\C_n^+}$.
Fix $\xi(x)$ and $\delta \in \C_n^+$ such that $x \notin \FV{\delta}$.
By $\LEM{\C_n^+}$, we have $\delta \lor \neg \delta$.
In former case, $\delta \to \exists x \xi(x)$ implies $\exists x \xi(x) $, which implies $\exists x(\delta \to \xi(x) )$.
In the latter case, we have $\delta \to \xi(x)$, and hence, $\exists x (\delta \to \xi(x))$.
Therefore, in any case, we have $(\delta \to \exists x \xi(x)) \to \exists x (\delta \to \xi(x))$. 

$(\ref{item: ->exists rule restricted to Dnn} \Rightarrow \ref{item: Sn-LEM} )$:
We show that for all $i\leq n$, $\HA +\eqref{item: ->exists rule restricted to Dnn}$
proves $\LEM{\Sigma_i}$ by induction on $i$.
The base case is trivial.
For the induction step, assume $i+1 \leq n$ and that $\HA +\eqref{item: ->exists rule restricted to Dnn}$ proves $\LEM{\Sigma_i}$.
We show $\exists \ol{x}  \vp(\ol{x}) \lor \neg \exists \ol{x} \vp(\ol{x})$ for $ \vp(\ol{x}) \in \Pi_i$ inside $\HA +\eqref{item: ->exists rule restricted to Dnn}$.
Now we have $ \exists \ol{x} \vp(\ol{x} ) \in \E_{i+1}^+ \subset \E_n^+ \subset \C_n^+$.
Therefore, by \eqref{item: ->exists rule restricted to Dnn}, we have
\begin{equation}
    \label{eq: formula from ->exists rile for Dnn}
\exists \ol{x}  \left( \exists \ol{x} \vp(\ol{x} ) \to \vp(\ol{x} )\right).\tag{c}
\end{equation}
Since $\neg \exists \ol{x} \vp(\ol{x}) $ is equivalent to $\forall \ol{x} \neg \vp(\ol{x})$ over $\HA$ and  $\forall \ol{x} \left( \vp(\ol{x} ) \lor \neg \vp(\ol{x} ) \right)$ is derived from $\LEM{\Sigma_i}$, \eqref{eq: formula from ->exists rile for Dnn} implies \eqref{eq: equivalent of the instance of Sn-LEM}, which is equivalent to $\exists \ol{x}  \vp(\ol{x}) \lor \neg \exists \ol{x} \vp(\ol{x})$ over $\HA$.

$(\ref{item: Dnn-LEM} \Rightarrow \ref{item: forall v rule restricted to Dnn})$:
We reason in $\HA + \LEM{\C_n^+}$.
Fix $\xi(x)$ and $\delta \in \C_n^+$ such that $x \notin \FV{\delta}$.
By $\LEM{\C_n^+}$, we have $\delta \lor \neg \delta$.
In former case, we have $\forall x \xi(x) \lor \delta$ trivially.
In the latter case, $\forall x (\xi(x) \lor \delta)$ implies $\forall x \xi(x) $, which implies $\forall x \xi(x) \lor \delta$.
Therefore, in any case, we have $\forall x (\xi(x) \lor \delta) \to \forall x \xi(x) \lor \delta$. 

$(\ref{item: forall v rule restricted to Dnn} \Rightarrow \ref{item: Sn-LEM})$:
We show that for all $i\leq n$, $\HA +\eqref{item: forall v rule restricted to Dnn} $
proves $\LEM{\Sigma_i}$ by induction on $i$.
The base case is trivial.
For the induction step, assume $i+1 \leq n$ and that $\HA +\eqref{item: forall v rule restricted to Dnn}$ proves $\LEM{\Sigma_i}$.
We show $\exists \ol{x}  \vp(\ol{x}) \lor \neg \exists \ol{x} \vp(\ol{x})$ for $ \vp(\ol{x}) \in \Pi_i$ inside $\HA +\eqref{item: forall v rule restricted to Dnn}$.
By  $\LEM{\Sigma_i}$, we have  $\forall \ol{x} \left( \neg \vp(\ol{x} ) \lor  \vp(\ol{x} ) \right)$, and hence,
\begin{equation}
\label{eq: cons of Si-LEM}
    \forall \ol{x} \left(\neg \vp(\ol{x} ) \lor  \exists \ol{x} \vp(\ol{x} ) \right). \tag{d}
\end{equation}
Now we have $ \exists \ol{x} \vp(\ol{x} ) \in \E_{i+1}^+ \subset \E_n^+ \subset \C_n^+$.
Therefore, by \eqref{item: forall v rule restricted to Dnn} and \eqref{eq: cons of Si-LEM}, we have $ \exists \ol{x} \vp(\ol{x} ) \lor \forall \ol{x} \neg \vp(\ol{x} )$, equivalently, $ \exists \ol{x} \vp(\ol{x} ) \lor \neg \exists \ol{x} \vp(\ol{x} )$.
\end{proof}

Based on the above observations, we define hierarchical semi-classical prenex normalization as follows:

\begin{defn}
\label{def: rules of qn}
$\vp \qn \psi$ denotes that $\psi$ is obtained from $\vp$ by substitutions with respect to the rules $(\exists \to)$, $(\forall \to )_{n}$ if $n\neq 0$, $(\to \exists)_{n}$, $(\to \forall )$, $(\exists \land)$,  $(\forall \land)$, $(\land \exists )$,  $( \land \forall)$, $(\exists \lor)$,  $(\forall \lor)_{n}$, $(\lor \exists )$,  $( \lor \forall)_{n}$,
$(\exists \text{-}var)$ and $(\forall \text{-}var)$ to subformula occurrences of which proper subformulas are in prenex normal form subsequently in finite many times, where the rules  $(\exists \to)$, $(\to \forall)$, $(\exists \land)$, $(\forall \land )$, $(\land \exists)$, $(\land \forall)$, $(\exists \lor)$, $(\lor \exists )$, $(\exists\text{-var})$ and $(\forall\text{-var})$ are given in Table \ref{table: PN},
and the others are the following:
\begin{center}
\begin{tabular}{rcccl}
$(\forall \to)_{n}$: & $\forall x \xi(x) \to \delta$ &$\rightsquigarrow$ & $\exists x(\xi(x) \to \delta)$ & with $\xi(x) \in \U_n^+$ and $x\notin \FV{\delta};$\\[2pt]
$(\to \exists )_{n}$: & $\delta \to \exists  x\, \xi(x)$ & $\rightsquigarrow$ & $\exists  x\, (\delta \to \xi(x))$ & where $\delta\in \C_n^+$ and $x\notin \FV{\delta};$\\[2pt]
$(\forall \lor)_{n}$: & $\forall x \, \xi(x) \lor \delta$ &$\rightsquigarrow$ & $\forall x\, (\xi(x) \lor \delta)$ &where $\delta\in \C_n^+$ and $x\notin \FV{\delta};$\\[2pt]
$(\lor \forall)_{n}$: & $\delta \lor \forall x \, \xi(x)$ & $\rightsquigarrow$ & $\forall x\, (\delta \lor \xi(x))$  &where $\delta\in \C_n^+$ and $x\notin \FV{\delta}.$ \\
\end{tabular}
\end{center}
\end{defn}

\begin{remark}
    If $\vp \qn \psi$, then $\FV{\vp} = \FV{\psi}$.
\end{remark}

The following proposition asserts that the reflexivity and transitivity hold for $\qn$.   
\begin{prop}
\label{prop: reflexivity and transitivity of qn}
    The following hold$:$
    \begin{enumerate}
    \item 
    $\vp \qn \vp ;$
    \item 
    If $\vp \qn \psi$ and $\psi \qn \chi$, then $\vp \qn \chi$.
    \end{enumerate}    
\end{prop}
\begin{proof}
    Trivial by the definition of $\qn$.
\end{proof}

\begin{prop}
\label{prop: qk -> qk+k'}
    If $\vp \qn \psi$, then $\vp {\,\, \rightsquigarrow^{*}_{n+n'}\,\,} \psi$.
    \end{prop}
\begin{proof}
Trivial since $\U_n^+ \subset \U_{n+n'}^+$ and $\C_n^+ \subset \C_{n+n'}^+$.
\end{proof}

\begin{prop}
\label{lem: subformula app}
 If $\vp$ is a subformula of $\xi$ and $\vp \qn \psi$, then $\xi \qn \xi \br{\psi /\vp}$.
\end{prop}
\begin{proof}
    Fix $n$.
    We show our assertion by induction on the number of the application of the rules in $\vp \qn\psi$.
 If  $\vp \qn\psi$ with no application of the rules, then $\psi$ is $\vp$, and $\xi\br{\psi /\vp}$ is $\xi$.
 Therefore $\xi \qn \xi\br{\psi /\vp}$ by the reflexivity of $\qn$.

For the induction step, let $\vp \qn \psi$ with $m+1$ applications of the rules of $\qn$.
In addition, assume $\vp \qn\psi'$ with $m$ applications of the rules and $\psi' \qn \psi$ with $1$ application of the rules.
Then, by the induction hypothesis, we have $\xi \qn \xi\br{\psi'/\vp}$.
Since a subformula of $\psi'$ whose proper subformulas are in prenex normal form is also such a subformula of $\xi\br{\psi'/\vp}$,
we have $ \xi\br{\psi'/\vp} \qn \xi\br{\psi /\vp} $ by applying the rule applied for $\psi' \qn \psi$.
By the transitivity of $\qn$, we have $\xi \qn \xi\br{\psi /\vp} $.
\end{proof}

Propositions \ref{prop: reflexivity and transitivity of qn}, \ref{prop: qk -> qk+k'} and \ref{lem: subformula app} are frequently used without mentioning in the following proofs.

\begin{lem}
\label{lem: vp1 land vp2 in Sigmak, Pi_k}
    \begin{enumerate}
        \item 
        If $\vp_1, \vp_2 \in \Sigma_{k+1}^+$, then there exists $\psi\in \Sigma_{k+1}^+$ such that $(\vp_1 \land \vp_2) \qz \psi$.
               \item 
        If $\vp_1, \vp_2 \in \Pi_{k+1}^+$, then there exists $\psi\in \Pi_{k+1}^+$ such that $(\vp_1 \land \vp_2) \qz \psi$. 
    \end{enumerate}
\end{lem}
\begin{proof}
    By simultaneous induction on $k$.
    The base step (the case of $k=0$) is trivial by the rules $(\exists \land )$, $(\forall \land )$, $(\land \exists )$ and $(\land \forall )$ of $\qz$.
    For the induction step, assume that $k\geq 1$ and the clauses hold for $k-1$.
    For the first clause for $k$, let $\vp_1$ be $\exists \ol{x}\, \vp'_1$ and $\vp_2$ be $\exists \ol{x} \,\vp'_2$ where $\vp'_1, \vp'_2 \in \Pi_k^+$.
    By the induction hypothesis, there exist $\psi' \in \Pi_k^+$ such that $(\vp'_1 \land \vp'_2)\qz \psi'$.
    By the rules $(\exists \land)$ and $(\land \exists)$, we have
    \[
\vp_1 \land \vp_2 \qz \exists \ol{x} \exists \ol{y} (\vp'_1 \land \vp'_2) \qz \exists \ol{x}\exists \ol{y}\, \psi',
    \]
    which is in $\Sigma_{k+1}^+$.
    The second clause for $k$ is shown similarly by using $(\forall \land)$ and $(\land \forall)$ instead of $(\exists \land)$ and $(\land \exists)$.
\end{proof}

\begin{lem}
\label{lem: vp1 lor vp2 in Sigma_k/Pi_k}
    \begin{enumerate}
        \item \label{item: lor Sigma}
        If $k \leq n$ and either $\vp_1 \in \Sigma_k^+\cup \Pi_k^+ $ and $\vp_2 \in \Sigma_{k+i+1}^+$, or $\vp_1 \in \Sigma_{k+i+1}^+$ and $\vp_2 \in \Sigma_k^+\cup \Pi_k^+ $, then there exists $\psi\in \Sigma_{k+i+1}^+$ such that $(\vp_1 \lor \vp_2 )\qn \psi$.
               \item \label{item: lor Pi}
                       If $k \leq n$ and either $\vp_1 \in \Sigma_k^+\cup \Pi_k^+ $ and $\vp_2 \in \Pi_{k+i+1}^+$, or $\vp_1 \in \Pi_{k+i+1}^+$ and $\vp_2 \in \Sigma_k^+\cup \Pi_k^+ $, then there exists $\psi\in \Pi_{k+i+1}^+$ such that $(\vp_1 \lor \vp_2 )\qn \psi$.
    \end{enumerate}
\end{lem}
\begin{proof}
Fix $n$.
By simultaneous induction on $k$, we show \eqref{item: lor Sigma} and \eqref{item: lor Pi} for all $k, i,\vp_1$ and $\vp_2$.

The base step is trivial by the rules $(\exists \lor ), (\forall \lor )_n, (\lor \exists), (\lor \forall)_n$.
For the induction step, assume $ k>0$ and that the assertions hold for $k-1$.

\eqref{item: lor Sigma}
Assume $k\leq n$, $\vp_1 \in \Sigma_k^+ \cup \Pi_k^+$ and $\vp_2 \in \Sigma_{k+i+1}^+$.
If $\vp_1\in \Sigma_{k-1}^+ \cup \Pi_{k-1}^+$, then we are done by the induction hypothesis.
Next, we reason in the case of $\vp_1 \equiv \exists \ol{x} \vp'_1 $ with $\vp'_1 \in \Pi_{k-1}$.
Let $\vp_2 \equiv \exists \ol{y} \vp'_2$ with $\vp'_2 \in \Pi_{k+i}^+$.
By the rules of $\qn$, we have $ \exists \ol{x} \vp'_1 \lor \exists \ol{y} \vp'_2 \qn \exists \ol{x}  \exists \ol{y} (\vp'_1 \lor  \vp'_2) $.
By the induction hypothesis, there exists $\psi' \in \Pi_{k+i}^+$ such that $  \vp'_1 \lor  \vp'_2 \qn \psi'$.
Therefore, we have $\vp_1 \lor \vp_2 \qn   \exists \ol{x} \exists \ol{y} \,\psi'$, which is in $\Sigma_{k+i+1}^+$.
Next, we reason in the case of $\vp_1 \equiv \forall \ol{x} \vp'_1 $ with $\vp'_1 \in \Sigma_{k-1}$.
Let  $\vp_2 \equiv \exists \ol{y_1} \forall \ol{y_2} \dots Q \ol{y_{i+1}}  \vp'_2$ with $\vp'_2 \in \Sigma_{k}^+ \cup \Pi_k^+$ (where $ Q \ol{y_{i+1}}$ is $ \forall \ol{y_{i+1}} $ or $ \exists \ol{y_{i+1}}$ depending on whether $i$ is even or odd).
Since $\forall \ol{x} \vp'_1\in \Pi_k^+ \subset \U_k^+ \subset \C_n^+$, by the rules of $\qn$, we have $\vp_1\lor \vp_2 \qn  \exists \ol{y_1} \forall \ol{y_2} \dots Q \ol{y_{i+1}}  (\forall \ol{x} \vp'_1 \lor \vp'_2)$.
If  $ Q \ol{y_{i+1}} \equiv \forall \ol{y_{i+1}} $ and $\vp'_2\in \Sigma_k^+$, by the induction hypothesis, there exists $\psi'\in \Sigma_k^+$ such that $ \vp'_1 \lor \vp'_2 \qn \psi'$, and hence, we have
\[
\vp_1\lor \vp_2 \qn  \exists \ol{y_1} \forall \ol{y_2} \dots \forall \ol{y_{i+1}} \forall \ol{x}( \vp'_1 \lor \vp'_2)\qn \exists \ol{y_1} \forall \ol{y_2} \dots \forall \ol{y_{i+1}} \forall \ol{x}\, \psi',
\]
which is in $\Sigma_{k+i+1}^+$.
If  $ Q \ol{y_{i+1}} \equiv \exists \ol{y_{i+1}} $ and $\vp'_2\in \Pi_k^+$, by the induction hypothesis, there exists $\psi'\in \Pi_k^+$ such that $ \vp'_1 \lor \vp'_2 \qn \psi'$, and hence, we have
\[
\vp_1\lor \vp_2 \qn  \exists \ol{y_1} \forall \ol{y_2} \dots \exists \ol{y_{i+1}} \forall \ol{x}( \vp'_1 \lor \vp'_2)\qn \exists \ol{y_1} \forall \ol{y_2} \dots \exists \ol{y_{i+1}} \forall \ol{x}\, \psi',
\]
which is in $\Sigma_{k+i+1}^+$.
The case of that $k\leq n$, $\vp_1 \in \Sigma_{k+i+1}^+$ and $\vp_2 \in \Sigma_k^+ \cup \Pi_k^+$ is verified similarly.

One can also show \eqref{item: lor Pi} by using the induction hypothesis in a similar manner.
\end{proof}

\begin{cor}
\label{cor: vp1 lor vp2 in Sigma_k+1}
    If $\vp_1, \vp_2\in \Sigma_{k+1}^+$ and $k\leq n$, then there exists $\psi \in \Sigma_{k+1}^+$ such that $\vp_1 \lor \vp_2 \qn \psi$.
\end{cor}
\begin{proof}
    Let $\vp_1 \equiv \exists \ol{x} \vp'_1$ with $\vp'_1\in \Pi_k^+$.
    By Lemma \ref{lem: vp1 lor vp2 in Sigma_k/Pi_k}.\eqref{item: lor Sigma}, there exists $\psi\in \Sigma_{k+1}^+$ such that $\vp'_1 \lor \vp_2 \qn \psi$.
    Therefore, we have
    \[
\vp_1 \lor \vp_2 \qn \exists \ol{x} \,(\vp'_1 \lor \vp_2) \qn \exists \ol{x}\,\psi,
    \]
    which is in $\Sigma_{k+1}^+$.
 \end{proof}

\begin{cor}
\label{cor: vp1 lor vp2 in Sigma_k/Pi_k}
    \begin{enumerate}
        \item \label{item: lor Sigma k=n}
        If either $\vp_1 \in \Sigma_n^+\cup \Pi_n^+ $ and $\vp_2 \in \Sigma_{n+i+1}^+$, or $\vp_1 \in \Sigma_{n+i+1}^+$ and $\vp_2 \in \Sigma_n^+\cup \Pi_n^+ $, then there exists $\psi\in \Sigma_{n+i+1}^+$ such that $(\vp_1 \lor \vp_2 )\qn \psi$.
               \item \label{item: lor Pi k=n}
                       If either $\vp_1 \in \Sigma_n^+\cup \Pi_n^+ $ and $\vp_2 \in \Pi_{n+i+1}^+$, or $\vp_1 \in \Pi_{n+i+1}^+$ and $\vp_2 \in \Sigma_n^+\cup \Pi_n^+ $, then there exists $\psi\in \Pi_{n+i+1}^+$ such that $(\vp_1 \lor \vp_2 )\qn \psi$.
    \end{enumerate}
\end{cor}
\begin{proof}
    Immediate from Lemma \ref{lem: vp1 lor vp2 in Sigma_k/Pi_k}.
\end{proof}

\begin{lem}
\label{lem: vp1 to vp2 in Sigma_k/Pi_k}
    \begin{enumerate}
          \item \label{item: to Sigma}
        If $k \leq n$, $\vp_1 \in \Sigma_k^+\cup \Pi_k^+ $ and $\vp_2 \in \Sigma_{k+i+1}^+$, then there exists $\psi\in \Sigma_{k+i+1}^+$ such that $(\vp_1 \to \vp_2 )\qn \psi$.
               \item \label{item: to Pi}
                       If $k \leq n$, $\vp_1 \in \Sigma_{k+1}^+ $ and $\vp_2 \in \Pi_{k+i+1}^+$, then there exists $\psi\in \Pi_{k+i+1}^+$ such that $(\vp_1 \to \vp_2 )\qn \psi$.
    \end{enumerate}
    
\end{lem}
\begin{proof}
Fix $n$.
By simultaneous induction on $k$, we show \eqref{item: to Sigma} and \eqref{item: to Pi} for all $k, i,\vp_1$ and $\vp_2$.

We first show \eqref{item: to Sigma} and \eqref{item: to Pi} in the case of $k=0$.
For \eqref{item: to Sigma}, assume $0\leq n$, $\vp_1 \in \Sigma_0^+\cup \Pi_0^+ $ and $\vp_2 \in \Sigma_{i+1}^+$.
Let $ \vp_2 \equiv \exists \ol{y_1} \forall \ol{y_2} \dots Q \ol{y_{i+1}} \vp'_2$ with $\vp'_2$ is quantifier-free.
Applying the rules of $\qn$ (even $\qz$), we have
\[
(\vp_1 \to \vp_2) \qn \exists \ol{y_1} \forall \ol{y_2} \dots Q \ol{y_{i+1}} (\vp_1 \to \vp'_2),
\]
which is in  $\Sigma_{i+1}^+$.
For \eqref{item: to Pi}, assume $0\leq n$, $\vp_1 \in \Sigma_1^+ $ and $\vp_2 \in \Pi_{i+1}^+$.
If $\vp_1 \in \Sigma_0^+\cup \Pi_0^+ $, as in the case of \eqref{item: to Sigma}, we have that there exists $\psi \in \Pi_{i+1}^+$ such that $(\vp_1 \to \vp_2 ) \qn \psi$.
Otherwise, there exists quantifier-free $\vp'_1 $ such that $\vp_1 \equiv \exists \ol{x}\, \vp'_1$.
By the rule $(\exists \to)$, we have $(\exists \ol{x} \vp'_1 \to \vp_2) \qn \forall \ol{x} (\vp'_1 \to \vp_2)$.
In addition, as in  the case of \eqref{item: to Sigma}, we have that there exists $\psi' \in \Pi_{i+1}^+$ such that $(\vp'_1 \to \vp_2 ) \qn \psi$.
Therefore, we have
\[
(\vp_1 \to \vp_2) \qn \forall \ol{x} (\vp'_1 \to \vp_2) \qn \forall \ol{x }\, \psi',
\]
which is in  $\Pi_{i+1}^+$.

For the induction step, assume $k>0$ and the assertions hold for $k-1$.
For \eqref{item: to Sigma}, assume $k\leq n$, $\vp_1 \in \Sigma_k^+\cup \Pi_k^+ $ and $\vp_2 \in \Sigma_{k+i+1}^+$.
Note $\Sigma_k^+\cup \Pi_k^+ = \Sigma_k^+\cup \Pi_k$.
First, we reason in the case of $\vp_1 \in \Sigma_k^+$.
Let $ \vp_2 \equiv \exists \ol{y}  \vp'_2$ with $\vp'_2\in \Pi_{k+i}^+$ 
Since $k-1 < k\leq n$, by \eqref{item: to Pi} of the induction hypothesis, there exists $\psi' \in \Pi_{k+i}^+$ such that $(\vp_1 \to \vp'_2) \qn \psi'$.
Since $\vp_1$ is in $\C_n^+$, using the rule $(\to \exists)_n$, we have
\[
(\vp_1 \to\vp_2) \qn \exists \ol{y} (\vp_1 \to \vp'_2) \qn  \exists \ol{y} \,\psi',
\]
which is in $\Sigma_{k+i+1}^+$.
Next, we reason in the case of $\vp_1 \equiv \forall \ol{x} \vp'_1$ with $\vp'_1 \in \Sigma_{k-1}$.
Let $ \vp_2 \equiv \exists \ol{y}  \vp'_2$ with $\vp'_2\in \Pi_{k+i}^+$.
By \eqref{item: to Pi} of the induction hypothesis, there exists $\psi' \in \Pi_{k+i}^+$ such that $(\vp'_1 \to \vp'_2) \qn \psi'$.
Since the proper subformulas of $\forall \ol{x} \vp'_1$ are in $\U_n^+$, using the rules $(\forall \to)_n$ and $(\to \exists)_n$, we have
\[
(\vp_1 \to\vp_2) \qn \exists \ol{x}  \exists \ol{y} (\vp'_1 \to \vp'_2) \qn  \exists \ol{x} \exists \ol{y} \,\psi',
\]
which is in $\Sigma_{k+i+1}^+$.

For \eqref{item: to Pi}, assume $k\leq n$, $\vp_1 \in \Sigma_{k+1}^+$ and $\vp_2 \in \Pi_{k+i+1}^+$.
Let $\vp_1 \equiv \exists \ol{x} \forall \ol{y} \vp'_1$ with $\vp'_1 \in \Sigma_{k-1}^+$ and $\vp_2 \equiv \forall \ol{z}  \vp'_2$ with $\vp'_2\in \Sigma_{k+i}^+$.
By \eqref{item: to Sigma} of the induction hypothesis, there exists $\psi' \in \Sigma_{k+i}^+$ such that $(\vp'_1 \to \vp'_2) \qn \psi'$.
Since the proper subformulas of $\forall \ol{y} \vp'_1$ are in $\U_n^+$, using the rules $(\exists \to)$ and $(\to \forall)$ and $(\forall \to)_n$, we have
\[
(\vp_1 \to\vp_2) \qn \forall \ol{x} \forall \ol{z} (\forall \ol{y} \vp'_1 \to \vp'_2)\qn \forall \ol{x} \forall \ol{z} \exists \ol{y} ( \vp'_1 \to \vp'_2) \qn \forall \ol{x} \forall \ol{z} \exists \ol{y}\, \psi',
\]
which is in $\Pi_{k+i+1}^+$.
\end{proof}

\begin{cor}
\label{cor: vp1 to vp2 in Sigma_k/Pi_k}
    \begin{enumerate}
          \item \label{item: to Sigma k=n}
        If $\vp_1 \in \Sigma_n^+\cup \Pi_n^+ $ and $\vp_2 \in \Sigma_{n+i+1}^+$, then there exists $\psi\in \Sigma_{n+i+1}^+$ such that $(\vp_1 \to \vp_2 )\qn \psi$.
               \item \label{item: to Pi k=n}
                       If $\vp_1 \in \Sigma_{n+1}^+ $ and $\vp_2 \in \Pi_{n+i+1}^+$, then there exists $\psi\in \Pi_{n+i+1}^+$ such that $(\vp_1 \to \vp_2 )\qn \psi$.
    \end{enumerate}
    
\end{cor}
\begin{proof}
    Immediate from Lemma \ref{lem: vp1 to vp2 in Sigma_k/Pi_k}.
\end{proof}

\begin{thm}
\label{thm: semi-classical PN}
    \begin{enumerate}
        \item 
        \label{item: Enk qn Sk}
        If $\vp\in \J_k^n$, then there exists $\psi\in \Sigma_k^+$ such that $\vp \qn \psi$.
        \item 
                \label{item: Unk qn Pk}
        If $\vp\in \R_k^n$, then there exists $\psi\in \Pi_k^+$ such that $\vp \qn \psi$.        
    \end{enumerate}
\end{thm}
\begin{proof}
Fix $n$.
By course-of-value induction on $k$, we show that for all $k$ and $\vp$, \eqref{item: Enk qn Sk} and \eqref{item: Unk qn Pk} hold.
The base step is trivial.
For the induction step, assume that \eqref{item: Enk qn Sk} and \eqref{item: Unk qn Pk} hold for all $\vp$ up to $k$, and show the assertion for $k+1$ by induction on the structure of formulas.
For a prime $\vp$, since $\vp \qz \vp$ and $\vp \in \Sigma_{k+1}^+ \cap \Pi_{k+1}^+$, we are done.
Suppose that the assertion holds for $\vp_1$ and $\vp_2$.

Suppose $ \vp_1 \land \vp_2 \in \J_{k+1}^n$.
By Lemma \ref{lem: land Enk and Unk}, we have $\vp_1, \vp_2 \in \J_{k+1}$.
By the induction hypothesis for $\vp_1$ and $\vp_2$, there exist $\psi_1, \psi_2\in \Sigma_{k+1}^+$ such that $\vp_1 \qn \psi_1$ and $\vp_2 \qn \psi_2$.
By Lemma \ref{lem: vp1 land vp2 in Sigmak, Pi_k}, there exists $\psi\in \Sigma_{k+1}^+$ such that $\psi_1 \land \psi_2 \qn \psi$.
Therefore,
we have
\[
\vp_1 \land \vp_2 \qn \psi_1 \land \psi_2 \qn \psi.
\]

In a similar manner, one can also show that if $\vp_1 \land \vp_2 \in \R_{k+1}^n$, there exists $\psi\in \Pi_{k+1}^+$ such that $\vp_1 \land \vp_2 \qn \psi$.

Suppose $ \vp_1 \lor \vp_2 \in \J_{k+1}^n$.

\medskip
\noindent
{\bf Case of $k\leq n$:}
By Lemma \ref{lem: lor Enk and Unk}, we have $\vp_1, \vp_2 \in \J_{k+1}^n$.
By the induction hypothesis for $\vp_1$ and $\vp_2$, there exist $\psi_1 , \psi_2 \in \Sigma_{k+1}^+$ such that $\vp_1 \qn \psi_1$ and $\vp_2 \qn \psi_2$.
By Corollary \ref{cor: vp1 lor vp2 in Sigma_k+1}, there exists $\psi \in \Sigma_{k+1}^+$ such that $\psi_1 \lor \psi_2 \qn \psi$.
Therefore, we have
\[
\vp_1 \lor \vp_2 \qn \psi_1 \lor \psi_2 \qn  \psi .
\]

\medskip
\noindent
{\bf Case of $k> n$:}
By Lemma \ref{lem: lor Enk and Unk}, we have that $\vp_1 \in \J_{k+1}^n$ and $\vp_2 \in \J_{n+1}^n$, or $\vp_1 \in \J_{n+1}^n$ and $\vp_2 \in \J_{k+1}^n$.
Without loss of generality, assume that $\vp_1 \in \J_{k+1}^n$ and $\vp_2 \in \J_{n+1}^n$.
By the induction hypothesis for $\vp_1$, there exists $\psi_1\in \Sigma_{k+1}^+$ such that $\vp_1 \qn \psi_1$.
On the other hand, by the induction hypothesis for $n+1$ (note $n+1 \leq k)$,  there exists $\psi_2\in \Sigma_{n+1}^+$ such that $\vp_2 \qn \psi_2$.
Let $\psi_2\equiv \exists \ol{y} \psi_2'$ with $\psi_2' \in \Pi_n^+$.
By Corollary \ref{cor: vp1 lor vp2 in Sigma_k/Pi_k}, there exists $\psi' \in \Sigma_{k+1}^+$ such that $\psi_1 \lor \psi_2' \qn \psi'$.
Therefore, we have
\[
\vp_1 \lor \vp_2 \qn  \psi_1 \lor \exists \ol{y}\psi_2' \qn \exists \ol{y}(\psi_1 \lor \psi_2') \qn \exists \ol{y} \psi',
\]
which is in $\Sigma_{k+1}^+$.

Suppose $ \vp_1 \lor \vp_2 \in \R_{k+1}^n$.

\medskip
\noindent
{\bf Case of $k< n$:}
By Lemma \ref{lem: lor Enk and Unk}, we have $\vp_1, \vp_2 \in \R_{k+1}^n$.
By the induction hypothesis for $\vp_1$ and $\vp_2$, there exist $\psi_1,  \psi_2 \in \Pi_{k+1}^+$ such that $\vp_1 \qn \psi_1$ and $\vp_2 \qn \psi_2$.
Since $k+1\leq n$ and $\Pi_{k+1}^+ \subset \Pi_n^+ \subset \C_n^+$, by applying the rules $(\exists \lor), (\forall \lor)_n, (\lor \exists)$ and $ (\lor \forall)_n$ of $\qn$, we have $\psi \in \Pi_{k+1}^+$ such that $\psi_1 \lor \psi_2 \qn \psi$.
Therefore, we have
\[
\vp_1 \lor \vp_2 \qn \psi_1 \lor \psi_2 \qn \psi.
\]

\medskip
\noindent
{\bf Case of $k\geq n$:}
Then, by the construction of the class $\R_{k+1}^n$, 
(i) $\vp_1 \lor \vp_2 \in \D_k^n$, (ii) $\vp_1 \in \R_{k+1}^n$ and $\vp_2 \in \D_n^n$, or (iii) $\vp_1 \in \D_n^n$ and  $\vp_2 \in \R_{k+1}^n$.

\noindent
Case of (i):
By the induction hypothesis for $k$, there exists $\psi \in \Sigma_k^+ \cup \Pi_k^+ \subset \Pi_{k+1}^+$ such that $\vp_1 \lor \vp_2 \qn \psi$.

\noindent
Case of (ii):
By the induction hypothesis for $\vp_1$, there exists $\psi_1\in \Pi_{k+1}^+$ such that $\vp_1 \qn \psi_1$.
On the other hand, by the induction hypothesis for $n$ (note $n \leq k)$,  there exists $\psi_2\in \Sigma_n^+ \cup \Pi_n^+$ such that $\vp_2 \qn \psi_2$.
Since $n<k+1 $, by Corollary \ref{cor: vp1 lor vp2 in Sigma_k/Pi_k},
there exists $\psi \in \Pi_{k+1}^+$ such that $\psi_1 \lor \psi_2 \qn \psi$.
Therefore, we have
\[
\vp_1 \lor \vp_2 \qn \psi_1 \lor \psi_2 \qn  \psi.
\]

\noindent
Case of (iii): Similar to the case of (ii).

Suppose $ \vp_1 \to \vp_2 \in \J_{k+1}^n$.

\medskip
\noindent
{\bf Case of $k< n$:}
By Lemma \ref{lem: to Enk and Unk}, we have $\vp_1 \in \R_{k+1}^n$ and $ \vp_2 \in \J_{k+1}^n$.
By the induction hypothesis for $\vp_1$ and $\vp_2$, there exists $\psi_1 \in \Pi_{k+1}^+$ and $\psi_2 \in \Sigma_{k+1}^+$ such that $\vp_1 \qn \psi_1$ and $\vp_2 \qn \psi_2$.
Let $\psi_1 \equiv \forall \ol{x} \psi_1'$ with $\psi_1' \in \Sigma_{k}^+$.
By Lemma \ref{lem: vp1 to vp2 in Sigma_k/Pi_k}, there exists $\psi' \in \Sigma_{k+1}^+$ such that $(\psi_1' \to \psi_2) \qn \psi'$.
Since $k+1 \leq n$ and the proper subformulas of $\forall \ol{x} \psi_1' $ are in $\U_n^+$, by the rule $(\forall \to)_n$ of $\qn$, we have
\[
(\vp_1 \to \vp_2) \qn (\forall \ol{x} \psi_1' \to \psi_2) \qn \exists \ol{x} (\psi_1' \to \psi_2) \qn \exists \ol{x} \psi',
\]
which is in $\Sigma_{k+1}^+$.

\medskip
\noindent
{\bf Case of $k\geq n$:}
Then, by the construction of the class $\J_{k+1}^n$,
(i) $\vp_1 \to\vp_2 \in \D_k^n$, or (ii) $\vp_1 \in \D_n^n$ and $\vp_2 \in \J_{k+1}^n$.

\noindent
Case of (i):
By the induction hypothesis for $k$, there exists $\psi \in \Sigma_k^+ \cup \Pi_k^+ \subset \Sigma_{k+1}^+$ such that $(\vp_1 \to \vp_2 )\qn \psi$.

\noindent
Case of (ii):
By the induction hypothesis for $\vp_2$, there exists $\psi_2\in \Sigma_{k+1}^+$ such that $\vp_2 \qn \psi_2$.
On the other hand, by the induction hypothesis for $n$ (note $n \leq k)$, there exists $\psi_1\in \Sigma_n^+ \cup \Pi_n^+$ such that $\vp_1 \qn \psi_1$.
Since $n<k+1$, by Corollary \ref{cor: vp1 to vp2 in Sigma_k/Pi_k}, there exists $\psi \in \Sigma_{k+1}^+$ such that $(\psi_1 \to \psi_2) \qn \psi$. 
Therefore, we have
\[
(\vp_1 \to \vp_2) \qn (\psi_1 \to \psi_2) \qn \psi.
\]

Suppose $ \vp_1 \to \vp_2 \in \R_{k+1}^n$.

\medskip
\noindent
{\bf Case of $k\leq n$:}
By Lemma \ref{lem: to Enk and Unk}, we have $\vp_1 \in \J_{k+1}^n$ and $ \vp_2 \in \R_{k+1}^n$.
By the induction hypothesis for $\vp_1$ and $\vp_2$, there exist $\psi_1 \in \Sigma_{k+1}^+$ and $\psi_2 \in \Pi_{k+1}^+$ such that $\vp_1 \qn \psi_1$ and $\vp_2 \qn \psi_2$.
By Lemma \ref{lem: vp1 to vp2 in Sigma_k/Pi_k}, there exists $\psi \in \Pi_{k+1}^+$ such that $(\psi_1 \to \psi_2) \qn \psi$.
Therefore, we have
\[
(\vp_1 \to \vp_2) \qn ( \psi_1 \to \psi_2) \qn \psi.
\]

\noindent
{\bf Case of $k> n$:}
By Lemma \ref{lem: to Enk and Unk}, we have $\vp_1 \in \J_{n+1}^n$ and $ \vp_2 \in \R_{k+1}^n$.
By the induction hypothesis for $\vp_2$, there exists $\psi_2\in \Pi_{k+1}^+$ such that $\vp_2 \qn \psi_2$.
On the other hand, by the induction hypothesis for $n+1$ (note $n+1 \leq k)$, there exists $\psi_1\in \Sigma_{n+1}^+ $ such that $\vp_1 \qn \psi_1$.
Since $n+1<k+1$, by Corollary \ref{cor: vp1 to vp2 in Sigma_k/Pi_k}, there exists $\psi \in \Pi_{k+1}^+$ such that $(\psi_1 \to \psi_2) \qn \psi$. 
Therefore, we have
\[
(\vp_1 \to \vp_2) \qn (\psi_1 \to \psi_2) \qn \psi.
\]

Suppose $ \exists \ol{x}\vp_1 \in \J_{k+1}^n$.
By Lemma \ref{lem: exists Enk and Unk}, we have $\vp_1 \in \J_{k+1}^n$.
By the induction hypothesis for $\vp_1$, there exists $\psi_1 \in \Sigma_{k+1}^+$ such that $\vp_1 \qn \psi_1$.
Therefore, we have $\exists \ol{x}\vp_1 \qn \exists \ol{x}\psi_1$, which is in $\Sigma_{k+1}^+$.

Suppose $ \exists \ol{x}\vp_1 \in \R_{k+1}^n$.
By Lemma \ref{lem: exists Enk and Unk}, we have $\exists \ol{x}\vp_1 \in \J_{k}^n$.
By the induction hypothesis for $k$, there exists $\psi\in \Sigma_k^+$ such that $\exists  \ol{x}\vp_1 \qn \psi$.
Since $\Sigma_k^+ \subset \Pi_{k+1}^+$, we are done.

As in the cases of $ \exists \ol{x}\vp_1 \in \J_{k+1}^n$ and $ \exists \ol{x}\vp_1 \in \R_{k+1}^n$, one can show that if $ \forall \ol{x}\vp_1 \in \J_{k+1}^n$, there exists $\psi\in \Sigma_{k+1}^+$ such that $\forall \ol{x}\vp_1 \qn \psi $, and also that if $ \forall \ol{x}\vp_1 \in \R_{k+1}^n$, there exists $\psi\in \Pi_{k+1}^+$ such that $\forall \ol{x}\vp_1 \qn \psi $, respectively.
\end{proof}

\begin{remark}
    In the proof of Theorem \ref{thm: semi-classical PN}, it is possible to use Lemma \ref{lem: lor Enk and Unk} in the case for $\vp_1 \lor \vp_2 \in \R_{k+1}^n$ and Lemma \ref{lem: to Enk and Unk} in the case for $\vp_1 \to \vp_2 \in \J_{k+1}^n$, instead of appealing to the construction of our classes $\J_{k+1}^n$ and $\R_{k+1}^n$.
    But we choose the latter because the number of case distinctions fewer.
\end{remark}

\begin{lem}
    \label{lem: Enk+1 is closed under qn-rule}
    If $\psi \in \J_{k+1}^n$ and $\vp \rightsquigarrow \psi$ with respect to each rule of $\qn$ in Definition \ref{def: rules of qn}, then $\vp \in \J_{k+1}^n$.
\end{lem}
\begin{proof}
If $k<n$, since $\J_{k+1}^n =\E_{k+1}^+$ (see Proposition \ref{prop: k>=n => Ek+=Enk and Uk+=Unk}) and the rule of $\qn$ are those of $\DTA$, our assertion follows from \cite[Lemma 12]{FK24}.
Then it suffices to consider each rule of $\qn$ with assuming that $k\geq n$. 
In the following proof, we assume that $x \notin \FV{\delta}$ and often suppress the variable $x$ in $\xi(x)$ for notational simplicity.

$(\exists \to):$
Suppose that $\forall x(\xi \to \delta) \in \J_{k+1}^n$.
By Lemma \ref{lem: forall Enk and Unk}, we  have $\xi \to \delta \in \R_{k}^n$ and $k>0$.
Now $k-1\leq n \, (\leq k)$ or $k-1>n$.
By Lemma \ref{lem: to Enk and Unk}, we have that $\xi \in \J_{k}^n$ and $\delta\in \R_{k}^n$ in the former case, and also that $\xi \in \J_{n+1}^n$ and $\delta\in \R_{k}^n$ in the latter case.
In each case, by the construction of the class $\R_{k}^n$, we have $\exists x \xi \to \delta \in \R_{k}^n \subset \J_{k+1}^n$.

$(\forall \to)_n:$
Suppose that $\exists x(\xi \to \delta) \in \J_{k+1}^n$, $\xi \in \U_n^+$ and $n\neq 0$.
By Lemma \ref{lem: exists Enk and Unk}, we  have $\xi \to \delta \in \J_{k+1}^n$.
By Lemma \ref{lem: to Enk and Unk}, we have $\delta\in \J_{k+1}^n$.
On the other hand, by our assumption, we have $\xi \in \U_n^+ = \R_n^n$, and hence, $\forall x \xi \in \R_n^n \subset \D_n^n$.
Then, by the construction of the class $\J_{k+1}^n$ for $k \geq n$, we have $\forall x \xi \to \delta \in \J_{k+1}^n $.

$(\to \exists)_n:$
Suppose that $\exists x(\delta \to \xi ) \in \J_{k+1}^n$ and $\delta \in \C_n^+$.
By Lemma \ref{lem: exists Enk and Unk}, we  have $\delta \to \xi \in \J_{k+1}^n$.
By Lemma \ref{lem: to Enk and Unk}, we have $\xi\in \J_{k+1}^n$.
On the other hand, by our assumption, we have $\delta \in \C_n^+ = \D_n^n$ (cf. Remark \ref{rem: classes C+}).
Then, by the construction of the class $\J_{k+1}^n$ for $k \geq n$, we have $ \delta \to \exists x \xi \in \J_{k+1}^n $.

$(\to \forall):$
Suppose that $\forall x(\delta \to \xi ) \in \J_{k+1}^n$.
By Lemma \ref{lem: forall Enk and Unk}, we have $\delta \to \xi \in \R_{k}^n$ and $k>0$.
Now $k-1\leq n \, (\leq k)$ or $k-1>n$.
By Lemma \ref{lem: to Enk and Unk}, we have that $\delta \in \J_{k}^n$ and $\xi \in \R_{k}^n$ in the former case, and also that $\delta \in \J_{n+1}^n$ and $\xi \in \R_{k}^n$ in the latter case.
In each case, by the construction of the class $\R_{k}^n$, we have $ \delta \to \forall x \xi \in \R_{k}^n\subset \J_{k+1}^n$.

The cases for $(\exists \land)$, $(\forall \land)$, $(\land \exists)$ and $(\land \forall)$ are verified by using Lemmas \ref{lem: exists Enk and Unk}, \ref{lem: forall Enk and Unk} and \ref{lem: land Enk and Unk} in a straightforward way.

$(\exists \lor):$
Suppose that $\exists x(\xi \lor \delta) \in \J_{k+1}^n$.
By Lemma \ref{lem: exists Enk and Unk}, we  have $ \xi \lor \delta \in \J_{k+1}^n$.
Now $k= n $ or $k>n$.
By Lemma \ref{lem: lor Enk and Unk}, we have that $\xi , \delta \in \J_{k+1}^n$ in the former case, and also that $\xi \in \J_{k+1}^n$ and $\delta\in \J_{n+1}^n$, or $\xi \in \J_{n+1}^n$ and $\delta\in \J_{k+1}^n$ in the latter case.
In the former case, we have $\exists x \xi \in \J_{k+1}^n$, and hence, $\exists x \xi \lor \delta \in \J_{k+1}^n$.
In the latter case, we have that $\exists x \xi \in \J_{k+1}^n$ and $\delta\in \J_{n+1}^n$, or $\exists x  \xi \in \J_{n+1}^n$ and $\delta\in \J_{k+1}^n$.
By the construction of the class $\J_{k+1}^n$ for $k > n$, we have $\exists x \xi \lor \delta \in \J_{k+1}^n$.

$(\forall \lor)_n:$
Suppose that $\forall x(\xi \lor \delta) \in \J_{k+1}^n$ and $\delta \in \D_n^n$.
By Lemma \ref{lem: forall Enk and Unk}, we  have $ \xi \lor \delta \in \R_{k}^n$ and $k>0$.
Now (i) $k-1<n$ (namely, $k=n$), (ii) $k-1 = n$, or (iii) $k-1>n$.
By Lemma \ref{lem: lor Enk and Unk},
we have that $\xi , \delta \in \R_{k}^n$ in the first case, and that $\xi \in \R_{k}^n$ and $\delta\in \D_{k-1}^n$, or $\xi \in \D_{k-1}^n$ and $\delta\in \R_{k}^n$ in the second case, and that $\xi \in \R_{k}^n$ and $\delta\in \J_{n+1}^n$, or $\xi \in \J_{n+1}^n$ and $\delta\in \R_{k}^n$ in the third case.
In the first case, we have $\forall x \xi \in \R_{k}^n$, and hence, $\forall x \xi \lor \delta \in \R_{k}^n$.
In the second case, we have $\forall x \xi \in \R_{k}^n$, and hence, $\forall x \xi \lor \delta \in \R_{k}^n$ by our assumption  $\delta \in \C_n^+$ and Remark \ref{rem: classes C+}.
In the third case, since $n+2 \leq k$, we again have $\forall x \xi \in \R_{k}^n$, and hence, $\forall x \xi \lor \delta \in \R_{k}^n$ by our assumption  $\delta \in \C_n^+$ and Remark \ref{rem: classes C+}.
Thus, in any case, we have $\forall x \xi \lor \delta \in \R_{k}^n$, and hence, $\forall x \xi \lor \delta \in \J_{k+1}^n$.

The cases for $(\lor \exists)$ and $(\lor \forall)_n$ are verified as in the cases for 
$(\exists \lor)$ and $(\forall \lor )_n$ respectively.

The cases for 
$(\exists\text{-var})$
and
$(\forall\text{-var})$
are trivial.
\end{proof}

\begin{lem}
    \label{lem: Unk+1 is closed under qn-rule}
    If $\psi \in \R_{k+1}^n$ and $\vp \rightsquigarrow \psi$ with respect to each rule of $\qn$ in Definition \ref{def: rules of qn}, then $\vp \in \R_{k+1}^n$.
\end{lem}
\begin{proof}
As in the proof of Lemma \ref{lem: Enk+1 is closed under qn-rule},
it suffices to consider each rule of $\qn$ with assuming that $k\geq n$. 
In the following proof, we assume that $x \notin \FV{\delta}$ and often suppress the variable $x$ in $\xi(x)$ for notational simplicity.

$(\exists \to):$
Suppose that $\forall x(\xi \to \delta) \in \R_{k+1}^n$.
By Lemma \ref{lem: forall Enk and Unk}, we  have $\xi \to \delta \in \R_{k+1}^n$.
Now $k\leq n$ (namely, $k=n$) or $k>n$.
By Lemma \ref{lem: to Enk and Unk}, we have that $\xi \in \J_{k+1}^n$ and $\delta\in \R_{k+1}^n$ in the former case, and also that $\xi \in \J_{n+1}^n$ and $\delta\in \R_{k+1}^n$ in the latter case.
In each case, by the construction of the class $\R_{k+1}^n$, we have $\exists x \xi \to \delta \in \R_{k+1}^n$.

$(\forall \to)_n:$
Suppose that $\exists x(\xi \to \delta) \in \R_{k+1}^n$, $\xi \in \U_n^+$ and $n\neq 0$.
By Lemma \ref{lem: exists Enk and Unk}, we have $\xi \to \delta \in \J_{k}^n$ and $k>0$.
Now $k-1< n$ (namely, $k=n$) or $k-1 \geq n$.
By Lemma \ref{lem: to Enk and Unk}, we have that $\xi \in \R_{k}^n$ and $\delta\in \J_{k}^n$ in the former case, and also that $\xi \in \J_{n+1}^n$ and $\delta\in \J_{k}^n$ in the latter case.
In the former case, we have $\forall x \xi \in \R_{k}^n$, and hence, $\forall x \xi  \to \delta \in \J_{k}^n \subset \R_{k+1}^n$.
In the latter case, by our assumption  $\xi \in \U_n^+$ and Proposition \ref{prop: k>=n => Ek+=Enk and Uk+=Unk}, we have $\forall x \xi \in \R_n^n \subset \D_n^n$, and hence, $\forall x  \xi \to \delta \in \J_{k}^n \subset \R_{k+1}^n$ by the construction of the class $\J_k^n$ for $k-1\geq n$. 

$(\to \exists)_n:$
Suppose that $\exists x(\delta \to \xi ) \in \R_{k+1}^n$ and $\delta \in \C_n^+$.
By Lemma \ref{lem: exists Enk and Unk}, we  have $\delta \to \xi \in \J_{k}^n$ and k>0.
Now $k-1< n$ (namely, $k=n$) or $k-1 \geq n$.
By Lemma \ref{lem: to Enk and Unk}, we have that $\delta \in \R_{k}^n$ and $\xi \in \J_{k}^n$ in the former case, and also that $\delta \in \J_{n+1}^n$ and $\xi \in \J_{k}^n$ in the latter case.
In the former case, we have $\exists x \xi \in \J_{k}^n$, and hence, $ \delta \to \exists x \xi \in \J_{k}^n \subset \R_{k+1}^n$.
In the latter case, by our assumption  $\delta \in \C_n^+$ and the construction of the class $\J_k^n$ for $k-1\geq n$, we have $ \delta \to \exists x \xi \in \J_{k}^n \subset \R_{k+1}^n$.

$(\to \forall):$
Suppose that $\forall x(\delta \to \xi ) \in \R_{k+1}^n$.
By Lemma \ref{lem: forall Enk and Unk}, we have $\delta \to \xi \in \R_{k+1}^n$.
Now $k\leq n$ (namely, $k=n$) or $k>n$.
By Lemma \ref{lem: to Enk and Unk}, we have that $\delta \in \J_{k+1}$ and $\xi \in \R_{k+1}^n $ in the former case, and also that $\delta \in \J_{n+1}^n$ and $\xi \in \R_{k+1}^n$ in the latter case.
In any case, we have that $\forall x \xi \in \R_{k+1}^n$, and hence, $\delta \to \forall x \xi \in \R_{k+1}^n$ by the construction of $\R_{k+1}^n$.

The cases for $(\exists \land)$, $(\forall \land)$, $(\land \exists)$ and $(\land \forall)$ are verified by using Lemmas \ref{lem: exists Enk and Unk}, \ref{lem: forall Enk and Unk} and \ref{lem: land Enk and Unk} in a straightforward way.

$(\exists \lor):$
Suppose that $\exists x(\xi \lor \delta) \in \R_{k+1}^n$.
By Lemma \ref{lem: exists Enk and Unk}, we  have $ \xi \lor \delta \in \J_{k}^n$ and $k>0$.
Now $k-1 \leq n \, (\leq k)$ or $k-1 >n$.
By Lemma \ref{lem: lor Enk and Unk}, we have that $\xi , \delta \in \J_{k}^n$ in the former case, and also that $\xi \in \J_{k}^n$ and $\delta\in \J_{n+1}^n$, or $\xi \in \J_{n+1}^n$ and $\delta\in \J_{k}^n$ in the latter case.
In the former case, we have $\exists x \xi , \delta \in \J_{k}^n$.
In the latter case, we have that $\exists x \xi \in \J_{k}^n$ and $\delta\in \J_{n+1}^n$, or $\exists x  \xi \in \J_{n+1}^n$ and $\delta\in \J_{k}^n$.
In each case, by the construction of the class $\J_{k}^n$, we have $\exists x \xi \lor \delta \in \J_{k}^n\subset \R_{k+1}^n$.

$(\forall \lor)_n:$
Suppose that $\forall x(\xi \lor \delta) \in \R_{k+1}^n$ and $\delta \in \C_n^+$.
By Lemma \ref{lem: forall Enk and Unk}, we  have $ \xi \lor \delta \in \R_{k+1}^n$.
Now $k= n$ or $k>n$.
By Lemma \ref{lem: lor Enk and Unk},
we have that $\xi \in \R_{k+1}^n$ and $\delta\in \D_{k}^n$, or $\xi \in \D_{k}^n$ and $\delta\in \R_{k+1}^n$ in the former case, and that $\xi \in \R_{k+1}^n$ and $\delta\in \J_{n+1}^n$, or $\xi \in \J_{n+1}^n$ and $\delta\in \R_{k+1}^n$ in the latter case.
In the former case, we have $\forall x \xi \in \R_{k+1}^n$, and hence, $\forall x \xi \lor \delta \in \R_{k+1}^n$ by our assumption  $\delta \in \C_n^+$ and Remark \ref{rem: classes C+}.
In the latter case, since $\R_{n+2}^n \subset \R_{k+1}^n$, we again have $\forall x \xi \in \R_{k+1}^n$, and hence, $\forall x \xi \lor \delta \in \R_{k+1}^n$ by our assumption  $\delta \in \C_n^+$ and Remark \ref{rem: classes C+}.

The cases for $(\lor \exists)$ and $(\lor \forall)_n$ are verified as in the cases for 
$(\exists \lor)$ and $(\forall \lor )_n$ respectively.

The case for $(\exists\text{-var})$ and $(\forall\text{-var})$ are trivial.
\end{proof}

\begin{remark}
    Our technical Lemmas \ref{lem: lor Enk and Unk} and \ref{lem: to Enk and Unk} seem to be essential for the proofs of Lemmas \ref{lem: Enk+1 is closed under qn-rule} and \ref{lem: Unk+1 is closed under qn-rule}.
\end{remark}

\begin{lem}
    \label{lem: classical occurrence substitution on Enk and Unk}
    If $n> k$ and $\xi \qn \xi'$, then the following hold$:$
    \begin{enumerate}
        \item \label{item: classical occurrence substitution on Enk}
        if $\vp' \equiv \vp\br{\xi' / \xi}\in \J_{k+1}^n$, then $\vp\in \J_{k+1}^n;$
           \item \label{item: classical occurrence substitution on Unk}
        if $\vp' \equiv \vp\br{\xi' / \xi}\in \R_{k+1}^n$, then $\vp\in \R_{k+1}^n.$
    \end{enumerate}
\end{lem}
\begin{proof}
    Since $\J_{k+1}^n =\E_{k+1}^+$ and $\R_{k+1}^n =\U_{k+1}^+$ for $k$ and $n$ such that $n> k$ (see Proposition \ref{prop: k>=n => Ek+=Enk and Uk+=Unk}) and the rules of $\qn$ are those of $\DTA$, our lemma follows from \cite[Lemma 13]{FK24}.
\end{proof}

\begin{lem}
\label{lem: occurrence substitution on Enk and Unk}
    If  $n\leq  k$ and $\xi \qn \xi'$, then the following hold$:$
    \begin{enumerate}
        \item \label{item: occurrence substitution on Enk}
        if $\vp' \equiv \vp\br{\xi' / \xi}\in \J_{k+1}^n$, then $\vp\in \J_{k+1}^n;$
           \item \label{item: occurrence substitution on Unk}
        if $\vp' \equiv \vp\br{\xi' / \xi}\in \R_{k+1}^n$, then $\vp\in \R_{k+1}^n.$
    \end{enumerate}
\end{lem}
\begin{proof}
    By course-of-value induction on $k$.
Since the base step is verified as in the induction step with some modification, we present only the proof of the induction step.
For the induction step, assume $k>0$ and that for any $k'\leq k-1$, any formula $\vp$, any $n\leq k'$ and any formulas $\vp', \xi$ and $\xi'$, if $\xi \qn \xi'$, then \eqref{item: occurrence substitution on Enk} and \eqref{item: occurrence substitution on Unk} hold for $k'$.
We show the assertion holds for $k$ by induction on the structure of formulas.
For a prime $\vp$, since $\vp \in \D_0^n$ for any $n$, we are done.
Assume that for any $n\leq k$ and any formulas $\vp', \xi$ and $\xi'$, if $\xi \qn \xi'$, then \eqref{item: occurrence substitution on Enk} and \eqref{item: occurrence substitution on Unk} hold for $\vp_1$ and $\vp_2$.
By Lemmas \ref{lem: Enk+1 is closed under qn-rule} and \ref{lem: Unk+1 is closed under qn-rule}, in each of the following cases, we may assume that $\xi$ is a proper subformula of $\vp$.

Let $\vp \equiv \vp_1 \land \vp_2$.
Let $n\leq k$.
Assume $\xi \qn \xi', \vp' \equiv \vp\br{\xi' / \xi}\in \J_{k+1}^n$ and that $\xi$ is a proper subformula of $\vp$ occurring in $\vp$.
Without loss of generality, let $\xi$ be a subformula of $\vp_1$ occurring in $\vp_1$.
Suppose $\vp' \equiv \vp_1 \br{\xi' / \xi} \land \vp_2 \in \J_{k+1}^n$.
By Lemma \ref{lem: land Enk and Unk}, we have $\vp_1 \br{\xi' / \xi} \in  \J_{k+1}^n$ and $\vp_2\in \J_{k+1}^n$.
By the induction hypothesis for $\vp_1$, we have $\vp_1 \in \J_{k+1}^n$, and hence, $\vp\equiv \vp_1 \land \vp_2 \in\J_{k+1}^n$.
In a similar way, one can also show the case for $\vp' \equiv \vp \br{\xi' / \xi} \in \R_{k+1}^n$.

Let $\vp \equiv \vp_1 \lor \vp_2$.
Let $n\leq k$.
Assume $\xi \qn \xi', \vp' \equiv \vp\br{\xi' / \xi}\in \J_{k+1}^n$ and that $\xi$ is a proper subformula of $\vp$ occurring in $\vp$.
Without loss of generality, let $\xi$ be a subformula of $\vp_1$ occurring in $\vp_1$.

\medskip
\noindent
Firstly, suppose $\vp' \equiv  \vp_1 \br{\xi' / \xi} \lor \vp_2 \in \J_{k+1}^n$.
If $n=k$,  by Lemma \ref{lem: lor Enk and Unk}, we have $\vp_1 \br{\xi' / \xi} , \vp_2 \in  \J_{k+1}^n$.
By the induction hypothesis for $\vp_1$, we have $\vp_1\in \J_{k+1}^n$, and hence, $\vp\equiv \vp_1 \lor \vp_2 \in\J_{k+1}^n$.
If $n <k$, by Lemma \ref{lem: lor Enk and Unk}, we have $\vp_1 \br{\xi' / \xi}\in  \J_{k+1}^n$ and $ \vp_2 \in \J_{n+1}^n$, or $\vp_1 \br{\xi' / \xi}\in  \J_{n+1}^n$ and  $\vp_2 \in  \J_{k+1}^n$.
In the former case, by the induction hypothesis for $\vp_1$, we have $\vp_1\in \J_{k+1}^n$, and hence, $\vp\equiv \vp_1 \lor \vp_2 \in\J_{k+1}^n$.
In the latter case, by the induction hypothesis for $n$ ($<k$), we have $\vp_1\in \J_{n+1}^n$, and hence, $\vp\equiv \vp_1 \lor \vp_2 \in\J_{k+1}^n$.

\medskip
\noindent
Secondly, suppose  $\vp' \equiv  \vp_1 \br{\xi' / \xi} \lor \vp_2 \in \R_{k+1}^n$.
Since $n\leq k$, by the construction of the class $\R_{k+1}^n$, (i) $ \vp_1 \br{\xi' / \xi} \lor \vp_2 \in \D_{k}^n $, (ii) $ \vp_1 \br{\xi' / \xi}  \in \R_{k+1}^n$ and $\vp_2 \in \D_{n}^n$, or (iii) $\vp_1 \br{\xi' / \xi} \in \D_{n}^n$ and $ \vp_2  \in \R_{k+1}^n$.
In the second case,  by the induction hypothesis for $\vp_1$, we have $\vp_1\in \R_{k+1}^n$, and hence, $\vp\equiv \vp_1 \lor \vp_2 \in\R_{k+1}^n$.
In the third case,
by Lemma \ref{lem: classical occurrence substitution on Enk and Unk},
we have $\vp_1\in \D_n^n$, and hence, $\vp\equiv \vp_1 \lor \vp_2 \in\R_{k+1}^n$.
We reason in the first case, namely, the case of  $ \vp_1 \br{\xi' / \xi} \lor \vp_2 \in \D_{k}^n $.
If $n\leq  k-1$, by the induction hypothesis for $k-1$, we have $\vp\equiv \vp_1 \lor \vp_2 \in \D_k^n \subset \R_{k+1}^n$.
If $n=k$, by Lemma \ref{lem: classical occurrence substitution on Enk and Unk}, we again have $\vp\equiv \vp_1 \lor \vp_2 \in \D_k^n \subset  \R_{k+1}^n$.

Let $\vp \equiv \vp_1 \to \vp_2$.
Let $n \leq k$.
Assume $\xi \qn \xi', \vp' \equiv \vp\br{\xi' / \xi}\in \J_{k+1}^n$ and that $\xi$ is a proper subformula of $\vp$ occurring in $\vp$.

\medskip
\noindent
{\bf Case of that $\xi$ is a subformula of $\vp_1$ occurring in $\vp_1$:}

\medskip
\noindent
Firstly, suppose $\vp' \equiv  \vp_1 \br{\xi' / \xi} \to \vp_2 \in \J_{k+1}^n$.
By the construction of the class $\J_{k+1}^n$, $ \vp_1 \br{\xi' / \xi} \to \vp_2 \in \D_k^n$, or $\vp_1 \br{\xi' / \xi} \in \D_n^n$ and $\vp_2 \in \J_{k+1}^n$.
In the latter case, by Lemma \ref{lem: classical occurrence substitution on Enk and Unk}, we have $\vp_1\in \D_n^n$, and hence, $\vp\equiv \vp_1 \to \vp_2 \in \J_{k+1}^n$.
We reason in the former case, namely, the case of $ \vp_1 \br{\xi' / \xi} \to \vp_2 \in \D_k^n$.
If  $n \leq k-1$, by  the induction hypothesis for $k-1$, we have $\vp\equiv \vp_1 \to \vp_2 \in \D_k^n \subset \J_{k+1}^n$.
If $n=k$,  by Lemma \ref{lem: classical occurrence substitution on Enk and Unk}, we again have $\vp\equiv \vp_1 \to \vp_2 \in \D_k^n \subset  \J_{k+1}^n$.

\medskip
\noindent
Secondly, suppose $\vp' \equiv  \vp_1 \br{\xi' / \xi} \to \vp_2 \in \R_{k+1}^n$.
By Lemma \ref{lem: to Enk and Unk}, we have that $\vp_1 \br{\xi' / \xi} \in \J_{n+1}^n$ and $\vp_2 \in \R_{k+1}^n$.
If $n=k$, by the induction hypothesis for $\vp_1$, we have $\vp_1 \in \J_{n+1}^n$.
If $n<k$, by the induction hypothesis for $n$ ($<k$),  we again have $\vp_1 \in \J_{n+1}^n$.
Therefore, we have $\vp \equiv \vp_1 \to \vp_2 \in \R_{k+1}^n$.

\medskip
\noindent
{\bf Case of that $\xi$ is a subformula of $\vp_1$ occurring in $\vp_2$:}

\medskip
\noindent
Firstly, suppose $\vp' \equiv  \vp_1  \to \vp_2\br{\xi' / \xi} \in \J_{k+1}^n$.
By the construction of the class $\J_{k+1}^n$, $ \vp_1  \to \vp_2\br{\xi' / \xi} \in \D_k^n$, or $\vp_1  \in \D_n^n$ and $\vp_2\br{\xi' / \xi} \in \J_{k+1}^n$.
In the former case, as in the case that $\xi$ is a subformula of $\vp_1$ occurring in $\vp_1$, we have $\vp\equiv \vp_1 \to \vp_2 \in \D_{k}^n \subset \J_{k+1}^n$.
In the latter case, by the induction hypothesis for $\vp_2$, we have $\vp_2 \in \J_{k+1}^n$, and hence, $\vp\equiv \vp_1\to \vp_2 \in \J_{k+1}^n$.

\medskip
\noindent
Secondly, suppose $\vp' \equiv  \vp_1 \to \vp_2  \br{\xi' / \xi} \in \R_{k+1}^n$.
By  Lemma \ref{lem: to Enk and Unk}, we have that $\vp_1  \in \J_{n+1}^n$ and $\vp_2\br{\xi' / \xi} \in \R_{k+1}^n$.
By the induction hypothesis for $\vp_2$, we have $\vp_2 \in \R_{k+1}^n$, and hence, $\vp \equiv \vp_1 \to \vp_2 \in \R_{k+1}^n$.

Let $\vp \equiv \exists x \vp_1 $.
Let $n\leq k$.
Assume $\xi \qn \xi', \vp' \equiv \vp\br{\xi' / \xi}\in \J_{k+1}^n$ and that $\xi$ is a proper subformula of $\vp$ occurring in $\vp$.
Then $\xi$ is a subformula of $\vp_1$ occurring in $\vp_1$.

\medskip
\noindent
Firstly, suppose $\vp' \equiv  \vp \br{\xi' / \xi} \equiv \exists x \left( \vp_1 \br{\xi' / \xi} \right)\in \J_{k+1}^n$.
By Lemma \ref{lem: exists Enk and Unk}, we have $\vp_1 \br{\xi' / \xi}\in \J_{k+1}^n$.
By the induction hypothesis for $\vp_1$, we have $\vp_1 \in \J_{k+1}^n$, and hence, $\exists x \vp_1 \in \J_{k+1}^n$.

\medskip
\noindent
Secondly, suppose $\vp' \equiv \exists x  \vp_1 \br{\xi' / \xi} \in \R_{k+1}^n$.
By Lemma \ref{lem: exists Enk and Unk}, we have $\exists x \vp_1 \br{\xi' / \xi} \in \J_{k}^n$ and $k>0$.
If $n\leq k-1$, by the induction hypothesis for $k-1$, we have $\vp \equiv \exists x \vp_1 \in \J_k^n \subset \J_{k+1}^n$.
If $n=k$, by Lemma \ref{lem: classical occurrence substitution on Enk and Unk}, we again have $\vp \equiv \exists x \vp_1 \in \J_k^n \subset \J_{k+1}^n$.

The case for $\vp \equiv \forall x \vp_1 $ can be verified as in the case for $\vp \equiv \exists x \vp_1 $ by using Lemma \ref{lem: forall Enk and Unk} instead of Lemma \ref{lem: exists Enk and Unk}.
\end{proof}

\begin{remark}
    In the proof of Lemma \ref{lem: occurrence substitution on Enk and Unk}, it is possible to appeal to the construction of our classes $\J_{k+1}^n$ and $\R_{k+1}^n$ instead of using Lemmas \ref{lem: land Enk and Unk}, \ref{lem: lor Enk and Unk} and Lemma \ref{lem: to Enk and Unk}.
    On the other hand, since the clause \eqref{item: or Unk} in Lemma \ref{lem: lor Enk and Unk} and the clause \eqref{item: to Enk} in Lemma \ref{lem: to Enk and Unk} in the case of $k>n$, are weaker than the consequences obtained from the constructions respectively,
    some appealing to the construction of the classes which appears in the proof of Lemma \ref{lem: occurrence substitution on Enk and Unk} seems to be necessary.
\end{remark}

\begin{thm}
\label{thm: the converse of semi-classical PN}
The following hold$:$
    \begin{enumerate}
        \item 
        if there exists $\psi\in \J_k^n$ such that $\vp \qn \psi$, then $\vp\in \J_k^n;$
        \item 
        if there exists $\psi\in \R_k^n$ such that $\vp \qn \psi$, then $\vp\in \R_k^n.$        
    \end{enumerate}
\end{thm}
\begin{proof}
For $k=0$, our assertions hold since no consequence of the rules of $\qn$ is quantifier-free.
For $k>0$, our assertions follows from Lemmas \ref{lem: classical occurrence substitution on Enk and Unk} and \ref{lem: occurrence substitution on Enk and Unk}.
\end{proof}

\begin{thm}
\label{thm: characterizations of semi-classical PN}
The following hold$:$
    \begin{enumerate}
        \item 
        $\vp\in \J_k^n$ if and only if there exists $\psi\in \Sigma_k^+$ such that $\vp \qn \psi;$
        \item 
         $\vp\in \R_k^n$ if and only if there exists $\psi\in \Pi_k^+$ such that $\vp \qn \psi.$       
    \end{enumerate}
\end{thm}
\begin{proof}
Our theorem follows from Theorems \ref{thm: semi-classical PN} and \ref{thm: the converse of semi-classical PN}
since $\Sigma_k^+\subset \J_k^n$ and $\Pi_k^+ \subset \R_k^n$.
\end{proof}

\begin{remark}
\label{rem: reduction vs derivability}
    As demonstrated in Theorem \ref{thm: equivalents of Sn-LEM over HA}, our semi-classical prenex normalization of degree $n$ corresponds to $\LEM{\Sigma_n}$. 
    However, semi-classical prenex normalization is a reduction procedure without any reference to the notion of derivability, and hence, there is a gap between the semi-classical prenex normalization of degree $n$ and the derivability relation in the presence of $\LEM{\Sigma_n}$.
    In particular, a formula which is equivalent to some prenex formula of degree $k$ intuitionistically with assuming $\LEM{\Sigma_n}$, may not be  transformed to some formulas in the same class by $\qn$.
    For example, $\forall x \QF{\vp} \lor \exists y \QF{\psi} \to \QF{\chi}$ (with appropriate variable conditions) is equivalent to $\exists x \forall y (\QF{\vp} \lor \QF{\psi} \to \QF{\chi})$ over $\HA + \LEM{\Sigma_1}$ since the latter theory proves the equivalence between
$    \forall x \QF{\vp} \lor \exists y \QF{\psi} \to \QF{\chi} $
and
 $   \left(\forall x \QF{\vp} \to \QF{\chi}  \right) \land   \left(\exists y \QF{\psi} \to \QF{\chi}  \right) $.
    On the other hand, $\forall x \QF{\vp} \lor \exists y \QF{\psi} \to \QF{\chi} $ cannot be transformed to any $\Sigma_2^+$-formula
by ${\, \rightsquigarrow^{*}_{1}\, }$.
\end{remark}

In the end of the paper, we show that $\LEM{\Sigma_n}$ is necessary to show the prenex normal form theorems for $\J^n_k$ and $\R^n_k$ for all $k$ in the context of intuitionistic arithmetic despite the fact that those classes do not contain all formulas transformed into some formulas in $\Sigma_k^+$ and $\Pi_k^+$  over $\HA +\LEM{\Sigma_n}$ (see Remark \ref{rem: reduction vs derivability}).
We recall the following notation employed in \cite{FK21}.

\begin{notation}
Let $T $ be an extension of $\HA$.
Let $\Gamma$ and $\Gamma'$ be classes of $\HA$-formulas.
  Then  $\PNFT{\Gamma}{\Gamma'} $ denotes the following statement$:$
  for any $\varphi \in \Gamma$, there exists $\varphi' \in \Gamma'$ such that $\FV{\vp}=\FV{\vp'}$ and
    $T  \vdash \varphi \leftrightarrow \varphi' $.
\end{notation}

\begin{prop}
\label{prop: PNFT for semi-classical classes in HA}
     For a semi-classical theory $T$ containing $\HA +\LEM{\Sigma_{n}}$, then $\PNFT{\J^n_k}{\Sigma_k}$ and  $\PNFT{\R^n_k}{\Pi_k}$ hold for all $k$.
\end{prop}
\begin{proof}
    Immediate from Theorems \ref{thm: semi-classical PN} and \ref{thm: equivalents of Sn-LEM over HA}.
\end{proof}

\begin{lem}
\label{lem: PNFT(Unn+1, Pn+1)=>SnLEM}
Let $T $ be a theory in-between $\HA+\LEM{\Sigma_{n-1}} $ and $\PA$.
If $\PNFT{\R^{n}_{n+1}}{\Pi_{n+1}}$, then $T \vdash \LEM{\Sigma_{n}}$.
\end{lem}
\begin{proof}
    Let $\vp \in \Sigma_n$.
Then $\vp \in \J^n_n$, and hence, $\vp \lor \neg \vp \in \R^n_{n+1}$ by the construction of $\R^n_{n+1}$.
By  $\PNFT{\R^{n}_{n+1}}{\Pi_{n+1}}$, there exists $\psi \in \Pi_{n+1}$ such that $\FV{\psi} =\FV{ \vp \lor \neg \vp }$ and $T \vdash \psi \leftrightarrow (\vp \lor \neg \vp )$.
Since $\PA \vdash \vp \lor \neg \vp $, we have $\PA \vdash \psi$.
Then, by the conservation theorem for semi-classical arithmetic (cf. \cite[Theorem 3.17]{FK23}), we have $\HA + \LEM{\Sigma_{n-1}} \vdash \psi$, and hence, $T \vdash \vp \lor \neg \vp$.
\end{proof}

\begin{thm}
\label{thm: characterization for PNFT for semiclassical classes in HA}
    Let $T $ be a theory in-between $\HA $ and $\PA$.
    The following are pairwise equivalent$:$
\begin{enumerate}
    \item
    \label{item: PNFT for Enk and Unk for all k}
$\PNFT{\J^{n}_{k}}{\Sigma_{k}}$ and $\PNFT{\R^{n}_{k}}{\Pi_{k}}$ hold for all $k$;
    \item 
    \label{item: PNFT for Unk for all k <= n}
    $\PNFT{\R^{n}_{k}}{\Pi_{k}}$ holds for all $k\leq n+1$;
\item 
\label{item: T |- SnLEM}
$T \vdash \LEM{\Sigma_{n}}$.
\end{enumerate}
\end{thm}
\begin{proof}
Implication $(\ref{item: T |- SnLEM} \Rightarrow \ref{item: PNFT for Enk and Unk for all k})$ is by Proposition \ref{prop: PNFT for semi-classical classes in HA}.
Implication $(\ref{item: PNFT for Enk and Unk for all k} \Rightarrow \ref{item: PNFT for Unk for all k <= n})$ is obvious.
In the following we show implication $(\ref{item: PNFT for Unk for all k <= n} \Rightarrow \ref{item: T |- SnLEM})$.
Assume that  $\PNFT{\R^{n}_{k}}{\Pi_{k}}$ holds for all $k\leq n+1$.
By Proposition \ref{prop: k>=n => Ek+=Enk and Uk+=Unk} and \cite[Lemma 7.2]{FK21}, we have $T \vdash \LEM{\Sigma_{n-1}}$.
Therefore, by Lemma \ref{lem: PNFT(Unn+1, Pn+1)=>SnLEM} and our assumption, we have $T \vdash \LEM{\Sigma_n}$.
\end{proof}

\begin{remark}
    We don't know whether it is also equivalent to $T\vdash \LEM{\Sigma_n}$ that $\PNFT{\J^{n}_{k}}{\Sigma_{k}}$ holds for all $k\leq n+1$ in Theorem \ref{thm: characterization for PNFT for semiclassical classes in HA}.
\end{remark}

\section*{Acknowledgements}
The first author was supported by JSPS KAKENHI Grant Numbers JP20K14354 and JP23K03205, and the second author by JP23K03200.
This work was also supported by the Research Institute for Mathematical
Sciences, an International Joint Usage/Research Center located in Kyoto
University.

\bibliographystyle{plain}
\bibliography{bibliography}

\end{document}